\documentclass[12pt,a4paper]{article}
\usepackage{amsmath}
\usepackage{amsthm}
\usepackage{amsfonts}
\usepackage{amssymb}
\usepackage{array}
\usepackage[OT2,OT1]{fontenc}
\newcommand\cyr{%
\renewcommand\rmdefault{wncyr}%
\renewcommand\sfdefault{wncyss}%
\renewcommand\encodingdefault{OT2}%
\normalfont
\selectfont}
\DeclareTextFontCommand{\textcyr}{\cyr}
\numberwithin{equation}{section}
\newtheorem{theorem}{Theorem}[section]
\newtheorem{cor}[theorem]{Corollary}
\newtheorem{lemma}[theorem]{Lemma}
\newtheorem{prop}[theorem]{Proposition}
\theoremstyle{definition}
\newtheorem{defi}{Definition}[section]
\newtheorem{rem}[defi]{Remark}

\newcommand\half{\frac{1}{2}}

\newcommand\ov{\overline}

\renewcommand\o{\varpi}

\newcommand\rhat{\widehat\rho}

\newcommand\be{\beta}

\newcommand\ka{\widehat \k}

\newcommand\w{\wedge}
\newcommand\g{\mathfrak g}

\newcommand\h{\mathfrak h}
\newcommand\ha{\widehat{\mathfrak h}}

\newcommand\n{\mathfrak n}
\newcommand\bb{\mathfrak b}
\newcommand\D{\Delta}
\renewcommand\l{\lambda}
\newcommand\Dp{\Delta^+}

\newcommand\Da{\widehat\Delta}
\newcommand\Pia{\widehat\Pi}
\newcommand\Dap{\widehat\Delta^+}
\newcommand\Wa{\widehat{W}}
\renewcommand\d{\delta}

\renewcommand\t{\otimes}
\renewcommand\a{\alpha}
\renewcommand\aa{\mathfrak a}
\renewcommand\b{\bb}

\renewcommand\k{\mathfrak k}

\newcommand\nat{\mathbb N}
\newcommand\ganz{\mathbb Z}

\newcommand\s{\sigma}
\renewcommand\L{\Lambda}

\renewcommand\aa{\mathfrak a}

\newcommand\C{\mathbb C}
\newcommand\R{\mathbb R}
\newcommand\si{\sigma}

\newcommand\What{\widehat W}

\renewcommand\ha{\widehat{\mathfrak h}}

\newcommand{\fa}{\mathfrak{a}}

\newcommand{\fg}{\mathfrak{g}}
\newcommand{\fh}{\mathfrak{h}}

\newcommand\p{\mathfrak p}

\newcommand{\fp}{\mathfrak{p}}
\newcommand{\rank}{{\rm rank}}

\renewcommand{\k}{\mathfrak{k}}
\newcommand{\asdim}{\text{\rm asdim}}

\newcommand{\comm}{{\rm comm}}

\begin{document}
\title{On the kernel of the affine Dirac operator}
\author{Victor~G. Kac\\ Pierluigi M\"oseneder Frajria\\ Paolo  Papi}

\maketitle
\centerline{\it To Ernest Borisovich Vinberg, on the occasion of his 70$^{th}$ birthday}
\vskip10pt

\begin{abstract}
Let $\fg$ be a finite-dimensional semisimple Lie algebra and $(\cdot\,,\cdot)$ its Killing form,  
$\sigma$ an elliptic automorphism of $\fg$, and $\fa$ a
$\sigma$-invariant reductive subalgebra of $\fg$, such that the restriction of
the form $(\cdot\,,\cdot)$ to $\fa$ is non-degenerate.  Let $\widehat{L} (\fg
,\sigma)$ and $\widehat{L}(\fa ,\sigma)$ be the associated twisted affine Lie
algebras and $F^\sigma (\fp)$ the $\sigma$-twisted Clifford
module over $\widehat{L} (\fa ,\sigma)$, associated to the
orthocomplement $\fp$ of $\fa$ in $\fg$.  Under suitable hypotheses on
$\sigma$ and $\fa$, we provide a general formula for the
decomposition of the kernel of the affine Dirac operator, acting
on the tensor product of an integrable highest weight $\widehat{L} (\fg
,\sigma)$-module and $F^\sigma (\fp)$, into irreducible $\widehat{L}
(\fa, \sigma)$-submodules.

As an application, we derive the decomposition of all level~1
integrable irreducible highest weight modules over orthogonal
affine Lie algebras with respect to the affinization of the
isotropy subalgebra of an arbitrary symmetric space.
\vskip 20pt

\end{abstract}

\section{Introduction}Let $\g$ be a finite-dimensional semisimple Lie 
algebra and
$(\cdot\,,\cdot)$ its Killing form,  
let $\sigma$ be an  elliptic automorphisms of $\g$ (i.e. $\sigma$ is 
semisimple 
with modulus one eigenvalues).
%We assume furthermore that $\s\mu=\mu\s$.
%, stabilizing a non-degenerate invariant bilinear 
%form $(\cdot\,,\cdot)$.  
Let $\aa$ be a $\si$-invariant reductive subalgebra of $\g$ such 
that $(\cdot\,,\cdot)$ is still 
nondegenerate when restricted to $\aa$. Let 
$\g^{\ov 0}$ 
and $\aa^{\ov 0}$ be the fixed point sets of 
$\s$ and $\s_{|\aa}$ respectively. Then we can choose   
a $\s$-invariant Cartan
subalgebra $\fh$ of $\fg$, such that $\fh_0 =\fh \cap \g^{\ov 0}$ is a 
%$\mu$-stable 
Cartan subalgebra of $\g^{\ov 0}$.   
The aim of this paper is to establish a 
formula describing the kernel of the Kac-Todorov affine Dirac operator, 
provided that there exists an elliptic automorphism $\mu$ of $\g$, commuting 
with $\s$, such that $\h_0$ is the centralizer of a Cartan subalgebra $\h_0^\mu$
in the algebra $(\g^{\ov 0})^\mu$, the fixed point set of $\mu$ in $\g^{\ov 0}$, and 
$\h_0^\mu$  is a Cartan subalgebra of $\aa^{\ov 0}$.
 This formula is a 
generalization of the formula proved in  \cite[Theorem 5.4]{KMP}, where it was 
assumed that the rank of $\aa^{\ov 0}$ equals the rank of $\g^{\ov 0}$.
%, while in the present work the rank of $\aa^{\ov 0}$ can be strictly 
%smaller than the rank of $\g^{\ov 0}$. 

To state the result precisely, 
let $L(\L)$ be a integrable  irreducible highest weight module over  the 
twisted affine Lie algebra $\widehat L(\g,\s)$, and let $F^{\si}(\p)$ be the 
$\sigma$-twisted Clifford module (see \eqref{clifford}), $\p$ being the 
orthocomplement of $\aa$ in $\g$. Let $D$ be the Kac-Todorov affine Dirac 
operator, which we regard as an operator on  $X=L(\L)\otimes F^{\si}(\p)$. 
%$\k_\eta$ 
 The main result of the present paper is the following:
\begin{theorem}\label{multiplet}  Assume that $(\L+\rhat_\si)_{|\h_0\cap\p}=0$.
Then the following decomposition into a direct sum of irreducible 
$\widehat L(\aa,\si)$-modules holds:
\begin{equation}\label{kerd} Ker\,(D)=2^{\lfloor \frac{\rank(\g^{\ov 0}))-
\rank(\aa^{\ov 0})+1}{2}\rfloor}\sum_{w\in \Wa'} V(\varphi_{\aa}^*
(w(\L+\rhat_\si))-\rhat_{\aa\,\si}).\end{equation}
Here  $\varphi_{\aa}^*,\,\rhat_\s,\rhat_{\aa\,\s}$,\, and $\Wa'$ are as 
defined in \eqref{vp}, \eqref{rhoaff}, and \eqref{mcr}  respectively.\par

\end{theorem}

Theorem \ref{multiplet} encompasses a long series of results which have their 
roots in the finite dimensional theory, as we presently explain. \par 
In his seminal paper \cite{P}, Parthasarathy pointed out a remarkable 
connection between the Atiyah-Bott  Dirac operator and the discrete series for 
a real semisimple group $G$ with an equal rank maximal compact subgroup $K$. 
His results can be recast in the following algebraic setting. Let $\g_0$ be 
the Lie algebra of $G$ and  $\g=\k\oplus\p$ be the (complexified) Cartan 
decomposition for the complexification $\g$ of $\g_0$. 
 Let $\h$ be a common Cartan subalgebra for $\g$ and $\k$. Fix a positive 
system $\Dp$ for the set of roots $\D_{\fg}$ of $(\fg, \fh)$ and let
%induce it on the roots of  $(\k,\h)$ by setting 
$\Dp_\k=\D_\k\cap \Dp$. Let $\rho,\rho_\k$ be the corresponding half sums 
of positive roots, and let 
$W,W_\k$ be the Weyl groups of $(\g,\h),\,(\k,\h)$, respectively. Denote 
finally by $W'$ the set of minimal right coset representatives of $W_\k$ in 
$W$.\par
 Parthasarathy proved that the spin representation $S$ of $\p$ decomposes, 
as a $\k$-module, as follows:
\begin{equation}\label{dspin} S=\bigoplus_{w\in W'}V(w(\rho)-\rho_\k)
\end{equation}
(cf . \cite[Lemma 2.2]{P}), and used this decomposition to calculate the 
kernel of the  Dirac operator. In turn he derived from this description an 
explicit realization of the discrete series \cite[Theorem 3]{P}. \par
Later on, Kostant realized that, upon a suitable modification of the Dirac 
operator, one could decompose directly its kernel into
collections of representations which he named {\it multiplets} 
(see the r.h.s. in formula \eqref{dec} below).  
At the same time he generalized the decomposition to the non-symmetric case, 
changing the setting as follows. Let $\aa$ be an equal rank  reductive 
subalgebra  of $\g$. 
%such that the restriction of the Killing form $B$ of $\g$ to $\aa$ is  nondegenerate. 
Consider the Dirac operator with cubic term
$${\not\!\partial}_{\g/\aa}=\sum_{i=1}^{\dim(\p)}z_i\otimes z_i+1\otimes v\in 
U(\g)\otimes Cl(\p),$$
where $\p$ is the orthocomplement of $\aa$ in $\g$, $\{z_i\}$ is an
% $B$-
orthonormal basis of $\p$  and $v$ is the image in the Clifford algebra 
$Cl(\p)$ of $\p$ of the fundamental
3-form $\omega\in \w^3(\p^*),$ $\omega(X,Y,Z) =  (X,[Y,Z])$ under the 
%Chevalley 
skewsymmetrisation map. 
%(which is the composition of the skewsymmetrisation map
%from the exterior algebra to the tensor algebra with the canonical projection 
%of the latter to the Clifford algebra0. 
This Dirac operator acts naturally on the $\aa$-module $L(\L)\otimes S$, where
$L(\L)$ is an irreducible finite dimensional highest weight $\g$-module and 
$S$ is, as above, the space of spinors for $\p$. Kostant proved in 
\cite{Kold} that 
 $Ker({\not\!\partial}_{\g/\aa})$  admits the following decomposition into 
irreducible $\aa$-modules
\begin{equation}\label{dec} Ker({\not\!\partial}_{\g/\aa})=
\bigoplus_{w\in W'}V(w(\L+\rho)-\rho_\aa).\end{equation}
(Notation is as above with $\aa$ in place of $\k$). 
Since in the symmetric case the kernel of the Dirac operator is the whole spin 
module $S$, putting $\L=0$ in \eqref{dec} one recovers formula \eqref{dspin}.
\par
 Decomposition formulas for the spin module $S$ into irreducible $\k$-modules, when $\k$ is the fixed point set of 
 an involution (but not necessarily of the same rank as $\g$) have been found
 in \cite{Wallach}, \cite{pan}, \cite{Han2}. The formula given in \cite[Lemma 9.3.2]{Wallach} is 
 \begin{equation}\label{W}
 S=2^{\lfloor\frac{\rank(\g)-\rank(\k)+1}{2} \rfloor}\sum_{P\in C(P_\k)} V(\rho_P-\rho_\k).\end{equation}
 Here $C(P_\k)$ is the set of positive subsets of $\D_\fg$, compatible with 
$\Dp_\k$. This means that  $P\in C(P_\k)$ if \begin{enumerate}
 \item $P$ is stable w.r.t the Cartan involution;
 \item
 if $\a\in \Dp_\k$ then there exists $\beta\in P$ such that the restriction of $\beta$ to a fixed Cartan subalgebra of $\k$ equals $\a$.
 \end{enumerate}
 Formula \eqref{W} does not show explicitly the emergence of multiplets. The obvious difficulty is that $W_\k$ is not naturally a subgroup of $W$.
 This problem has been overcome in  \cite{pan} with a case by case approach and, in  a uniform fashion, in \cite{Han2}. Let us describe the main 
 idea. Let $\mathfrak t$ be a Cartan subalgebra of $\k$, contained in
% we may assume that $\mathfrak t\subset 
$\h$. Let $W_\k$ be the Weyl group of
 $(\k, \mathfrak t)$ and $W_{\comm}$
the subgroup of $W$ formed by the elements commuting with the Cartan involution. It turns out that  $W_{\comm}$ is a Coxeter group containing $W_\k$ as a reflection subgroup, and if $W'$ is a set of minimal right coset representatives of 
$W_\k$ in $W_{\comm}$, one has \cite[Proposition 1.2]{Han2}
\begin{equation}\label{han}
 S=2^{\lfloor\frac{\rank(\g)-\rank(\k)+1}{2} \rfloor}\sum_{w\in W'} V(w(\rho)-\rho_\k).
\end{equation}\vskip5pt
The generalization of the previous results to the affine setting has many different aspects. One has first to remark that
in the infinite-dimensional case there are essentially two types of spaces of spinors for the affinization $\widehat{so(V)}$ of the special orthogonal algebra.
They are called basic+vector and spin representations, and correspond to certain (sums of) fundamental representations of  $\widehat{so(V)}$.
For any symmetric pair $(\g,\k)$, one has
an isotropy representation $\k\to so(\p)$ which gives rise to an embedding
$\ka\to \widehat{so(\p)}$. Therefore it is natural to investigate the  
decomposition  of the basic+vector and spin representations into irreducible  
$\ka$-modules.
This analysis has been performed in \cite{CKMP}, deepening previous work of Kac and Peterson \cite{KPNAS}.
The formulas we obtained, which can be considered the infinite-dimensional analogues of \eqref{dspin} and \eqref{han},  show the presence of multiplets, but fail to make it clear why multiplets appear.
On the other hand, formula \eqref{dec} has been generalized by Landweber 
\cite{land} to the affine case using a suitable analogue
 of the cubic  Dirac operator, still obtaining multiplets.\par
Formula  \eqref{kerd}  connects all these items in the setting of twisted 
affine Lie algebras and provides a general framework 
 for the emergence of multiplets. 
 \par
 The basic tool is the Kac-Todorov (cubic) affine Dirac operator $D$ 
(which was introduced in \cite{KacT} before \cite{land}). 
This operator was used 
in \cite{KMP} to obtain a generalization of \eqref{dec} in the twisted affine 
equal rank setting, therefore providing a conceptual explanation of the 
emergence of multiplets in the equal rank case.  
Indeed, in \cite[Theorem 5.4]{KMP} we proved that  that upon replacing  ${\not\!\partial}_{\g/\aa}$  by  $D$, $\g$  by the twisted loop algebra $\widehat L(\g,\si)$,  $\aa$ by $\widehat L(\aa,\si)$ and $L(\L)\otimes S$ by $X$,  one gets a decomposition formula which looks exactly like \eqref{dec}.\par
The missing element for  treating the non equal rank case 
was the identification of the affine analogue of 
$W_{\comm}$. 
It turns out that the correct choice is the subgroup $\widehat{W}_{\comm}$ of the  
Weyl group of $\widehat L(\g,\si)$ formed by the elements commuting with  $\mu$. The last  ingredient for the proof of  
\eqref{kerd} is Proposition \ref{newkumar}, which is  the non equal rank version of a standard result which goes back to $\n$-cohomology theory in the finite dimensional case.\par
Let us briefly describe the organization of the paper. After a thorough 
explanation of the setup in Section 2, we 
review in Section 3 some basic material on twisted affine Lie algebras. In  
particular we construct explicitly the root data of $\widehat L(\g,\si)$ 
in terms of $\sigma$ (cf. Proposition \ref{Pi}), also treating the case 
of semisimple $\g$. These results for arbitrary semisimple $\fg$ 
seem to be new (in \cite{Kac} only simple $\fg$ are treated).
% and of some independent interest.  
Section 4 
is  the most technical one. The upshot is the machinery of minimal coset 
representatives for the symmetry groups naturally appearing in the picture. 
It is a kind of affine analogue, in the framework of Steinberg's abstract 
approach to reflection groups, of the construction of \cite{Han2} outlined
above. The main results here are Corollary 
\ref{coxeter} and Lemma \ref{subgroup}.    
In Section 5 we prove Theorem \ref{multiplet} and in Section 7
we apply it to recover from a new point of view 
the decomposition formulas for the basic+vector and spin representations found 
in \cite{CKMP}. To accomplish this, we need a detailed analysis of the decomposition 
of  Clifford modules as representations of  orthogonal affine algebras. This is done in Section 6.
In Section 8 we deal with asymptotic dimensions 
of multiplets, providing formulas
for their signed sum in the equal rank case.  We are also able to determine the cases in which the
Dirac operator vanishes identically on $X$. In these cases we provide a formula for the sum 
of the asymptotic dimensions of the elements of the multiplet.
 \section{Setup}
For a finite-dimensional reductive Lie algebra $\fg$ over $\C$ with a given symmetric  non-degenerate invariant bilinear form (.,.) denote by $C_\fg$ the corresponding Casimir operator and let $2g_i$ $(i=1,\dots,T$) be the eigenvalues of $C_\g$ on $\g$, 
$\g_i$ being the corresponding eigenspaces.\par
Let $\si$ be an elliptic automorphism of $\g$ 
%with modulus $1$ eigenvalues
leaving the bilinear form $(\cdot,\cdot)$ invariant. For $j\in\R$, 
let $\ov j$ denote the class of $j$ modulo $\ganz$ and
$\g^{\ov j}$ the $\s$-eigenspace with the eigenvalue $e^{2\pi i j}$.  Set 
$$L(\g,\si)
=\sum\limits_{m\in \ov j}
(t^{m}\otimes \g^{\ov j}),\quad\widehat{L}(\g,\si)'=L(\g,\si)\oplus 
\sum_{i=0}^T\C K_i.$$ 
The latter is a central extension of the Lie algebra $L(\g,\si)$ with the Lie 
algebra bracket defined by
$$ [t^m\otimes  a,t^n\otimes b ]=t^{m+n}\otimes[a,b] + 
\d_{m,-n}m(a,b)K_i,\quad m,n\in\ \R\\, , a,b\in\g_i,
$$ $K_i$ being central elements.  
We  extend the Lie algebra $\widehat{L}(\g,\si)'$ by setting $$\widehat  
L(\g,\si)=\widehat{L}(\g,\si)'\oplus \C d,$$ where
$d$ is the derivation of $\widehat{L}(\g,\si)'$  such that $d(K_i)=0$ and $d$ 
acts as $t\frac{d}{dt}$ on
$L(\g,\si)$. The Lie algebra $\widehat L(\g,\si)$ is called
the $\si$-twisted affinization 
of $\g$ with respect to $(\cdot,\cdot)$.  

 Let $\mu$ be an elliptic automorphism of $\g$,
preserving the invariant bilinear form and commuting with $\si$. Then $\mu(\g^{\ov j})\subseteq \g^{\ov j}$. In particular $\mu$ 
induces an automorphism  of $\g^{\ov 0}$ (still denoted by $\mu$).
Consider the set $(\g^{\ov 0})^\mu$ of $\mu$-fixed points in $\g^{\ov 0}$.
Let $\h_0^\mu$ be a Cartan subalgebra of $(\g^{\ov 0})^\mu$ and let $\h_0$ be  centralizer of $\h_0^\mu$ in $\g^{\ov 0}$. Then $\h_0$ is a Cartan subalgebra of $\g^{\ov 0}$. Let $\D_0$ be the set of 
roots of $(\g^{\ov 0},\h_0)$ and fix a set $\Dp_0$ of positive roots that is $\mu$-stable.
Let $\Pi_0$ be the corresponding set of simple roots.
\par
Assume from now that $\g$ is semisimple and that  the form $(\cdot,\cdot)$ is a positive multiple of the Killing form. It follows that $C_\g$ acts on $\g$ as $2gI_\g$ with $g>0$. 
Let $\aa$ be a $\si$-stable reductive subalgebra of $\g$ such that the invariant form 
$(\cdot\,,\cdot)$ is still nondegenerate when restricted to $\aa$. Set
$\aa^{\ov j}=\aa\cap\g^{\ov j}$. Assume furthermore that $\h_0^\mu$ is a 
Cartan subalgebra of $\aa^{\ov 0}$. Let $\p$ be the orthocomplement of $\aa$ 
in $\g$.

 Let  $\D_\aa$ be the set of $\h_0^\mu$-roots for  $\aa^{\ov 0}$.   Then it 
is clear that  $\D_\aa$ is a subset ${\D_0}_{|\h_0^\mu}$. Moreover we can 
and do choose  $\Dp_\aa={\Dp_0}_{|\h_0^\mu}\cap \D_\aa$ as a set of positive 
roots for  $\D_\aa$.

Let  $\widehat L(\aa,\si)$ be the $\sigma$-twisted affinization (with  respect 
to $(\cdot,\cdot)_{|\aa}$)  of  $\aa$. 
Set, using standard notation, 
\begin{equation}\label{cartan}\ha=\h_0\oplus\C K\oplus\C d,\quad \ha^\mu=
\h_0^\mu\oplus \C K\oplus\C d,\end{equation}
($K$ being the central element corresponding to the unique eigenvalue of 
$C_\g$) and
\begin{equation}\label{level}\ha_\aa=\h_0^\mu \oplus\sum_i\C K_i\oplus\C d.\end{equation}
Let $\Da$ be the  set of $\ha$-roots of $\widehat L(\g,\si)$. As a set of 
positive roots for $\Da$ we choose 
\begin{equation}\label{possyst}
\Dap=\Dp_0\cup\{\a\in\Da\mid \a(d)>0\}.\end{equation} Analogously, if $\Da_\aa$ is the set 
of roots for $\widehat L(\aa,\si)$ then we choose 
$\Dap_\aa=\Dp_\aa\cup\{\a\in\Da_\aa\mid \a(d)>0\}$ as a set of positive roots. 
Let $\Pia_\sigma,\,\Pia_\aa$ be the corresponding sets of indecomposable roots.
 
 Let $\L_0$ be the element of $\ha^*$ defined  setting 
$\L_0(d)=\L_0(\h_0)=0$  and $\L_0(K)=1$.  Similarly, 
let $\L_0^i\in \ha_\aa^*$ defined 
by $\L_0^i(\h_0^\mu \oplus\C d)=0,\ \L_0^i(K_j)=\d_{ij}.$
Define also 
$\d\in\ha^*$  setting $\d(d)=1$ and $\d(\h_0)=\d(K)=0$.  Let $\d_\aa$  be 
the analogous element of
$\ha_\aa^*$ defined by  $\d_\aa(K_i)=0$ for all $i$, 
$\d_\aa(\h_0^\mu)=0,\,\d_\aa(d)=1$.

 In  \cite[Ch. 10, \S{} 5]{Helgason1} it is shown that $(\cdot,\cdot)_{|\h_0\times\h_0}$ is nondegenerate, thus we can define dually a form $(\cdot,\cdot)$ on $\h_0^*$. Extend $(\cdot,\cdot)$ to all of $\ha^*$ by setting $(\L_0,\d)=1$ 
and $(\L_0,\L_0)=(\d,\d)=(\d,\h_0)=(\L_0,\h_0)=0$. Let $\nu:\ha\to\ha^*$ be 
the isomorphism induced by the form $(\cdot,\cdot)$.
Write $\h_0=\h_0^\mu\oplus \h_\p$ with $\h_\p=\h_0\cap\p$. 
Regard $(\h_0^\mu)^*$ as a subspace of $\h_0^*$ by extending functionals to 
$\h_\p$ by zero.
In turn, we may view $(\ha^\mu)^*$ as a subspace of $\ha^*$. Notice that both 
$\L_0$ and $\d$ are in $(\ha^\mu)^*$, thus our formulas define also a bilinear 
form $(\cdot,\cdot)$ on $(\ha^\mu)^*$.

Set, as usual, $\rho_{\ov 0}=\half\sum_{\a\in\Dp_0}\a$, $\rho_{\aa\,\ov 0}=
\half\sum_{\a\in\Dp_\aa}\a$. Let 
$\D_{\ov j}$ be the set of $\h_0$-weights of $\g^{\ov j}$ and define, 
for $\ov j\ne \ov 0$
\begin{align}
&\rho_{\ov j}=
\half\sum_{\a\in\D_{\ov j}}(\dim\g^{\ov j}_\a)\a,&\rho_\si=
\sum_{0\le j<\frac{1}{2}}(1- 2j)\rho_{\ov j},
\\
&\rho_{\aa \ov j}=\half\sum\limits_{\a\in(\D_{\ov 
j})_{|\h_0^\mu}}(\dim\aa^{\ov j}_\a)\a,
&\rho_{\aa\,\si}=\sum\limits_{0\leq j<\half} (1-2j)\rho_{\aa \ov  j},
\\\label{rhoaff}
&\rhat_\si=\rho_\si+g\L_0,&\rhat_{
\aa\,\si}=\rho_{\aa\,\si}+\sum_ig_i\L_0^i.
\end{align} 
\vskip10pt
Set $\b=\h_0\oplus \n$ to be the Borel subalgebra of $\g^{\ov 0}$ 
corresponding to our choice of $\Dp_0$. Set $\n'=
\mathfrak  n+\sum_{j>0}(t^j\otimes\g^{\ov j})$. A 
$\widehat L(\g,\si)$-module $M$ is called a highest weight module with 
highest weight 
$\L\in\ha^*$ if there is a nonzero vector $v_\L\in M$ such that 
\begin{equation}\label{highest}\mathfrak n'(v_\L)=0,\,\ 
hv_\L=\L(h)v_\L \text{  for $h\in\ha$, }\
U(\widehat L(\g,\si))v_\L=M.\end{equation} 
Given a weight $\L$ in $\ha^*$, we
 let $L(\L)$ be the irreducible highest weight module for 
$\widehat L(\g,\si)$ with highest weight $\L$. 

Similarly, setting $\n'_\aa=\n'\cap\widehat L(\aa,\si)$, a highest weight 
module for $\widehat L(\aa,\si)$ with highest weight $\xi\in\ha_\aa^*$ is a 
$\widehat L(\aa,\si)$-module $N$ having  a nonzero vector $v_\xi\in N$ such 
that 
\begin{equation}\label{highesta}\mathfrak n'_\aa(v_\xi)=0,\,\ 
hv_\xi=\xi(h)v_\xi \text{  for $h\in\ha_\aa$, }\
U(\widehat L(\aa,\si))v_\xi=N.\end{equation} 
Given a weight $\xi$ in $\ha_\aa^*$, we
 let $V(\xi)$ be the irreducible highest weight module for $\widehat 
L(\aa,\si)$ with highest weight $\xi$.

 We retain  the setting of \cite{KMP}. In particular, by specializing to  
$A=\p$ in the construction  of \cite[Section 3.2]{KMP}, we obtain a 
Clifford module that we denote by $F^\si (\p)$
 (see also \eqref{clifford} below). This module is
 denoted by $F^{\ov\tau}(\ov\p)$ in \cite{KMP}.\par As observed in Remark 3.1 
of \cite{KMP}, if $M$ is a highest weight module for $\widehat L(\g,\si)$,
then the module $X=M\otimes F^\si (\p)$ acquires a natural action of 
$\widehat L(\aa,\si)$. Moreover, if $K$ acts on $M$ by $kI_M$, then $K_i$ 
acts on $X$ by $(k+g-g_i)I_X$.
 
 \begin{rem}\label{varfi}This last property can be restated as follows. 
Consider the map $\varphi_\aa:\ha_\aa\to\ha^\mu$, defined by 
\begin{equation}\label{vp}
\varphi_\aa(h)=h\ 
\text{if $h\in\h_0^\mu \oplus\C d$},\quad 
\varphi_\aa(K_i)=K\ 
\text{for all $i$}.\end{equation} 
Then, if $\xi$ is a $\widehat\h_\aa$-weight of $X$, 
$\xi+\rhat_\aa\in\varphi_\aa^*((\ha^\mu)^*)$.
\end{rem}

Let $D$ be the Kac-Todorov (relative) affine Dirac operator acting on $X$. 
This is the operator $(G_{\g,\aa})^X_0$ defined in  Section 4 of \cite{KMP}. 
It has
%$D$ is an operator acting on $X$ having 
the following properties:
\begin{enumerate}
\item $[D,a]=0$ for all $a\in\widehat L(\aa,\si)$.
\item If $N$ is a highest weight module over $\widehat L(\aa,\si)$ with 
highest weight $\xi$ occurring in $X$ and $v\in N$, then 
\begin{equation}\label{azionedirac}
D^2\cdot v=\left(\Vert \L+\rhat\Vert^2-\Vert(\varphi_\aa^*)^{-1}(\xi+\rhat_\aa)\Vert^2\right)v.
\end{equation}
\end{enumerate}
Notice that 2. above makes sense because of Remark \ref{varfi}.
\vskip10pt
\section{Twisted affine Lie algebras}In the rest of the paper we assume that $\g$ is semisimple.

 We now review the theory of twisted affine Lie algebras assuming   
$\sigma$  indecomposable. In particular we show that  $\widehat L(\g,\sigma)$ is an affine Kac-Moody Lie algebra. We follow the approach outlined in Section 8.8 of \cite{Kac} as exposed in  \cite[Ch. 10, \S{} 5]{Helgason1}.
 It is shown in \cite{BM} that there is a regular element of $\g$ that is 
fixed by $\sigma$. This in turn implies that in any Cartan subalgebra of $ \g^{\ov 0}$ (in particular $\h_0$) there  is a $\g$-regular element. Hence  its centralizer is a Cartan subalgebra $\h$ 
of $\g$ and  there is a positive system 
$\Phi^+$ for the set $\Phi$ of roots of 
$(\g,\h)$ having the property that the automorphism of $\Phi$ induced by 
$\sigma$ stabilizes $\Phi^+$. Let $\eta$ be the corresponding diagram automorphism of  $\g$
(cf. \cite[\S 8.1]{Kac}). We can write $\s=\eta e^{2\pi iad(h)}$ with $h\in\h_0$.
  Let $\Pia_\s=
\{\tilde\a_0,\tilde\a_1,\dots\}$ be the set of indecomposable roots in 
$\Da^+$. Set $\D$ to be the set of $\h_0$-weights of $\g$ and set
\begin{equation} \label{realh}
(\h_0)^*_{\R}=span_{\R}(\D).
\end{equation}
 Following \cite[Ch. 10, \S{} 5]{Helgason1} we have 
\begin{enumerate}
\item $\Pia_\s$ is finite having precisely $\dim\h_0+1$ elements.
\item $(\cdot,\cdot)_{|(\h_0)^*_{\R}\times(\h_0)^*_{\R}}$ is positive definite.
%\item $\Pia_\s$ is an indecomposable system of vectors.
\item The matrix $A=(a_{ij})$ where $a_{ij}=
\frac{2(\tilde\a_i,\tilde\a_j)}{(\tilde\a_j,\tilde\a_j)}$ is an indecomposable 
Cartan matrix of affine type.
\end{enumerate}

As shown in the proof of Lemma 5.10 i) of \cite{Helgason1}, this implies that 
$\widehat L(\g,\sigma)$ is isomorphic to the Kac-Moody algebra $\g(A)$. We denote by 
$\Wa_\s$ be its Weyl group.
 Remark that in 
\cite[Lemma 5.3]{KMP} Êwe have proved that 
\begin{equation}\label{coroot}
2\frac{(\rhat_\s,\a)}{(\a,\a)}=1\quad\forall\,\a\in\Pia_\s.
\end{equation}

\vskip5pt
We turn now to the determination of the set of simple roots of 
$\widehat L(\g,\sigma)$. 
As a first case we assume $\sigma=\eta$. 
 Let $\k_\eta$ be the fixed point set of $\eta$ in $\g$. Let $\Phi_\eta$ be 
its set of $\h_0$-roots.  Then $\Phi^+_\eta=\Phi_\eta\cap \Phi^+_{|\h_0}$ is 
a set of positive roots for $\Phi_\eta$. Let $\Pi_\eta=\{\a_1,\dots,\a_l\}$ 
be the corresponding set of simple roots.  
Let $r$ be the order of $\eta$ and set $\g_1$ to be the $\eta$-eigenspace in $\fg$ with the eigenvalue 
$e^{-2\pi i/r}$.

 \begin{lemma}\label{eta}
 The space $\g_1$ is irreducible as a $\k_\eta$-module. Moreover, 
if $\theta$ is its highest weight, then 
 \begin{equation}\label{rseta}
 \Pia_\eta=\{\frac{1}{r}\d-\theta\}\cup\Pi_\eta.
\end{equation}
 \end{lemma} 
 \begin{proof}It is clear that the simple roots of $\k_\eta$ are 
indecomposable in $\Phi^+_\eta$ hence $\Pi_\eta\subset\Pia_\eta$.  
By property 1 above we deduce that
 $
 \Pia_\eta\backslash\Pi_\eta$ has only one element that we denote by $\a_0$. In order to compute $\a_0$, we argue as in the proof of Lemma 5.10 of \cite{IMRN}. Set $\widehat{\mathfrak u}=\sum_{j>0}t^j\otimes \g^{\bar j}$ and $\widehat{\mathfrak u}^-=\sum_{j<0}t^j\otimes \g^{\bar j}$. Then $(\ha+\k_\eta)\oplus\widehat{\mathfrak u}$ and $(\ha+\k_\eta)\oplus\widehat{\mathfrak u}^-$ are the opposite parabolic subalgebras of $\widehat L(\g,\eta)$ corresponding to $\Pi_\eta$, whose nilradicals are $\widehat{\mathfrak u}$, $\widehat{\mathfrak u}^-$ respectively.  Let $\partial_p:\w^p \widehat{\mathfrak u}^-\to \w^{p-1} \widehat{\mathfrak u}^-$ be the boundary operator affording the Lie algebra homology $H_*(\widehat{\mathfrak u}^-)$. Remark that the Laplacian $L_1=\partial_1^*\partial_1+\partial_2\partial_2^*$ (here $\partial^*$ is the adjoint of $\partial$ w.r.t. the Hermitian form defined in \cite{Kumar}) is zero on $t^{-\frac{1}{r}}\otimes \g^{\ov{-1/r}}\subset\widehat{\mathfrak u}^-$, so that
 $t^{-\frac{1}{r}}\otimes \g^{\ov{-1/r}}$
   is a submodule of $H_1(\widehat{\mathfrak u}^-)$. On the other hand, if   
$\Wa'_\eta$ is a set of minimal right coset representatives of the Weyl group 
of $\k_\eta$
in  $\Wa_\eta$  then, by the Kostant-Garland-Lepowsky theorem (see e.g. 
\cite{Kumar}) and  \eqref{coroot}, we have
\begin{equation}\label{gl} H_1(\widehat{\mathfrak u}^-)\cong \bigoplus_{\substack{w\in \Wa'_\eta\\ \ell(w)=1}} V(w(\rhat_{\eta})-\rhat_{\eta})=\bigoplus_{\a\in\Pia_\eta\backslash\Pi_\eta} V(s_\a(\rhat_{\eta})-\rhat_{\eta})=V(-\a_0).\end{equation}
 Use now the identification $\g_1\cong t^{-\frac{1}{r}}\otimes \g^{\ov{-1/r}}$ as $\k_\eta$-modules to deduce at once that $\g_1$ is irreducible and that its highest weight $\theta$ is equal to $-(\a_0)_{|\h_0}$ and in turn that $\a_0=\frac{1}{r}\d-\theta$. \end{proof}
 
We now turn to the general case.
Let  $W_\eta$ be the Weyl group of $\k_\eta$, and consider
\begin{equation}\label{latticeM}
 M=span_\ganz(\frac{1}{r}W_\eta\theta^\vee),
 \end{equation} 
 a lattice in $\h_0$.
 Define, for $\a\in M$, the corresponding ``translation" in  $\ha^*$ by
 $$
 t_\a(\l)=\l+\l(K)\nu(\a)-(\l(\a)+\half|\a|^2\l(K))\d,
 $$
and the translation in  $\ha$ by the contragradient action:
 $$
 t_\a(h)=h+\d(h)\a-((h,\a)+\half|\a|^2\d(h))K.
 $$
As in \cite[Chapter 6]{Kac} one proves that if  $\a\in M$, then $t_\a\in\Wa_\eta$ and that 
 $
 \Wa_\eta=T\rtimes W_\eta
 $
where $T=span (t_\a)$. For this remark that  $s_\theta\in W_\eta$ since  $\theta$ is a multiple of a root of  $\k_\eta$  (see  \cite[Proposition 9.18 (a)]{Carter}).\par
Set $(\h_0)_\R=\oplus_{\a\in\Pi_\eta}\R\a^\vee$ and $\ha_\R=\R d\oplus \R K\oplus(\h_0)_\R$. As usual, exploit the  bijection
 $$\{h\in \ha_\R\mid\d(h)=1\}/\R K \to (\h_0)_\R$$
given by the natural projection to define an affine action on $(\h_0)_\R$. It is clear that $t_\l(h)=h+\l$ for $\l\in M$ and that the linear reflection w.r.t. $\frac{1}{r}\d-\theta$ maps to the reflection in the affine hyperplane $\theta=\frac{1}{r}$.  Lemma \ref{eta} implies that
\begin{equation}\label{caf}
C_f=\{h\in(\h_0)_\R\mid \a_i(h)\geq 0,\,i=1,\ldots,l,\,\theta(h)\leq \frac{1}{r}\}
\end{equation}  
 is a fundamental domain for this action. Therefore,
 if $\sigma=\eta e^{2\pi i \,ad(h)}$,  there exists  $w\in\Wa_\eta$ such that  $w(d+h)\equiv d+h'\mod \R K$ with  $\a_i(h')\ge 0$ and $\theta(h')\le \frac{1}{r}$.

Consider  the map $\a\mapsto w^{-1}(\a)+w^{-1}(\a)(h)\d$. This map is a bijection between the roots of $\widehat L(\g,\eta)$ and the roots of $\widehat L(\g,\sigma)$.

If we decompose $w^{-1}$ as  $w^{-1}=t_\l w',\,w'\in W_\eta$ then 
\begin{align*}
t_\l w'(\a)+t_\l w'(\a)(h)\d&=w'(\a)-w'\a(\l)\d+w'(\a)(h)\d\\&=w'(\a)+\a((w')^{-1}(h-\l))\d=w'(\a)+\a(h')\d.
\end{align*}
In particular, if we set  
\begin{equation}\label{si}
s_i=\a_i(h')\text{Ê for  $i=1,\ldots,l$ and } s_0=\frac{1}{r}-\theta(h')
,\end{equation} then we have 
$ w^{-1}(\a_i)+w^{-1}(\a_i)(h)\d=w'(\a_i)+s_i\d$ for $i=1,\ldots,l$ and $w^{-1}(\a_0)+w^{-1}(\a_0)(h)\d=-w'(\theta)+s_0\d$.
Hence
\begin{equation}\label{piaprime}\Pia'=\{s_0\d-w'(\theta),s_1\d+w'(\a_1),\dots,s_l\d+w'(\a_l)\}
\end{equation}
is a set of simple roots for $\widehat L(\g,\sigma)$. In particular, the set  
$\Pi'_0$ of the roots in $\Pia'$ such that  $s_i=0$ is a set of simple roots for $\g^{\ov 0}$. By finite dimensional theory (see \cite{Carter}), $w'(\theta),w'(\a_1),\ldots,w'(\a_l)$ are multiples of roots of $\k_\eta$. Hence there exists an element $w''\in W_\eta$ such that $w''(\Pi'_0)=\Pi_0$. Observe that $w^{-1}s_{\a_i}=t_\l s_{w'(\a_i)}w'$ if $i=1,\dots,l$ while $w^{-1}s_{\a_0}=t_{\l+(1/r)w'(\theta^\vee)}s_{w'(\theta)}w'$. If $s_i=0$ then we can substitute $w$ with $s_{\a_i}w$, thus we can choose $w$ so that $w^{-1}=t_{\l'}w''w'$. Hence we can assume that $\Pi_0'=\Pi_0$.

Since
$
\Da\cap(\sum_{\a\in\Pia'}\nat\a)\subset\Dap$ and 
$ \Da\cap(-\sum_{\a\in\Pia'}\nat\a)\subset-\Dap$ we may deduce that  $\Da\cap\sum_{\a\in\Pia'}\nat\a=\Dap.$ We have proven the following

\begin{prop}\label{Pi} Write $\s=\eta e^{2\pi i ad(h)}$ with $\eta$ a diagram automorphism and $h\in\h_0$. Choose $w\in\Wa_\eta$ such that $w(d+h)\equiv d+h'\mod \R K$ with $h'\in C_f$. Write $w^{-1}=t_\l w'$ with $\l\in M$ and $w'\in W_\eta$. Set $s_i=\a_i(h')$ for $i=1,\dots,l$ and $s_0=\frac{1}{r}-\theta(h')$. We can choose $w$ in such a way that, if 
\begin{equation}\label{sr}\Pia=\{s_0\d-w'(\theta),s_1\d+w'(\a_1),\dots,s_l\d+w'(\a_l)\},\end{equation}
then  the set of roots in $\Pia$ such that $s_i=0$ is equal to $\Pi_0$. 

Then the set $\Pia$ in \eqref{sr} is the set $\Pia_\s$ of simple roots of $\widehat L(\g,\s)$ corresponding to $\Dap$.\end{prop}

\section{Preparation on Weyl groups}
In the rest of the paper we assume that $\g$ is semisimple and that $\s$ is 
an elliptic automorphism (not necessarily indecomposable) of $\g$.

Recall that we introduced another automorphism $\mu$  of $\g$ and assumed that $\mu\s=\s\mu$.  Extend $\mu$ to $\ha^*$ by setting $\mu(\d)=\d$ and $\mu(\L^i_0)=\L^i_0$. 
Observe that $\mu$ induces an automorphism of the diagram of $\widehat L(\g,\s)$. Indeed $\mu(\Da)\subset\Da$, since $\mu\s=\s\mu$. Since we chose  $\Dp_0$ to be $\mu$-stable, we deduce  from \eqref{possyst} that  $\mu(\Dap)\subset \Dap$,  as desired.\par If $J$ is a $\mu$-orbit in $\Pia_\s$ and the root system $\D_J$ generated by $J$ is of finite type,  we let $(w_0)_J$ be the the longest element of the Weyl group of $\D_J$. If $\D_J$ is not of finite type, we set $(w_0)_J=1$. Let $\a_J=\frac{1}{|J|}\sum_{\a\in J}\a$. Clearly $\a_J=(\a_j)_{|\ha^\mu}$ for any $j\in J$.

From now on we assume for simplicity that the action induced by $\mu$ on the set of components of $\Pia_\s$ has a single orbit. 
\begin{lemma}\label{independent}
The set $\{\a_J\}$ is linearly independent, hence it is a basis of $(\h_0^\mu)^*\oplus\C\d$.
\end{lemma}
\begin{proof}
Let $\{\Gamma_1,\dots,\Gamma_t\}$ be the components of $\Pia_\s$ and set $\ha^e=\h_0\oplus\C K\oplus \sum_{i=1}^t \C d_i$ ($e$ stands for ``extended"). Then extending $\a\in\Gamma_i$ to an element $\a^e$ of $(\ha^e)^*$ by setting $\a^e(d_j)=\d_{ij}\a(d)$, we have that $(\Pia_\s)^e$ is linearly independent. Extend $\mu$ to $\ha^e$ in such a way that $\mu(d_i)=d_j$ if $\mu(\Gamma_i)=\Gamma_j$. Hence $\mu(\a^e)=\mu(\a)^e$. It is then clear that $\{(\a_J)^e\}$ is linearly independent. Set $d^e=\sum d_i$. By our assumption $(\ha^e)^\mu$ is $\C d^e \oplus \h_0^\mu\oplus \C K$. Now the map $K\mapsto K$, $d\mapsto d^e$, $h\mapsto h$ if $h\in\h_0^\mu$ is an isomorphisms between $\ha^\mu$ and $(\ha^e)^\mu$, thus its transpose is an isomorphism from $((\ha^e)^\mu)^*$ to $(\ha^\mu)^*$ that maps $(\a_J)^e$ to $\a_J$.
\end{proof}

 Set $$\Wa(\mu)=\{w\in\Wa_\s\mid w\mu=\mu w\}.$$ 
 Let  
$\widehat{W}_{\comm}$ 
be the group of linear transformations of $\ha^\mu$ generated by the set of 
reflections $S=\{s_\a\mid \a\in {(\Pia_\s)}_{|\ha^\mu}\}$. (Here, if $\a$ is isotropic, we mean $s_\a$ to be the identity). 
 If $w\in\Wa(\mu)$,
we let $\tilde w$ be the restriction of $w$ to $\ha^\mu$.\par

\begin{prop}\label{Wmu}
The map $w\mapsto \tilde w$ defines an isomorphism between $\Wa(\mu)$ and 
$\widehat{W}_{\comm}$. 
\end{prop}\begin{proof} First of all we prove that the map $w\mapsto \tilde w$ from $\Wa(\mu)$ to the set of linear maps on $\ha^\mu$ is injective. In fact, if $w\ne 1$ then there is $\a_j$ such that $w(\a_j)\in-\Dap$. Assume $\a_j\in J$. Since ${(\Dap)}_{|\ha^\mu}\cap(-\Dap)_{|\ha^\mu}=\emptyset$, we see that $\tilde w(\a_J)\ne\a_J$, hence $\tilde w\ne 1$.

Next we show that the image of the map contains $\Wa_\comm$. If $\D_J$ is of finite type, then arguing as in Proposition 9.17 of \cite{Carter}, we see that $(w_0)_J$ is in $\Wa(\mu)$ and that $\widetilde{(w_0)_J}=s_{\a_J}$. We now check that, if $\D_J$ is not of finite type then $\a_J$ is isotropic, so that $\widetilde{(w_0)_J}=s_{\a_J}$ in this case too. Fix a component $J_0$ of $J$ and let $\langle\mu\rangle_0$ be the subgroup of the cyclic subgroup $\langle\mu\rangle$ generated by $\mu$ that stabilizes $J_0$. Since $J$ is an orbit of $\langle\mu\rangle$, we have that $J_0$ is an orbit of $\langle\mu\rangle_0$. If $\mu^k$ is any generator of $\langle\mu\rangle_0$, we have that $\mu^k$ is a diagram automorphism of $J_0$ having a single orbit. Browsing Tables Aff 1--3 of \cite{Kac}, we see that $J_0$ is of type $A^{(1)}_l$ ($l\ge 1$). But in this case $\sum_{\a_j\in J_0}\a_j\in\C\d$, thus $\a_J$ is isotropic.\par
It remains to show that the image of the map is precisely $\Wa_\comm$. This is accomplished by showing that $\Wa(\mu)$ is generated by $$\{(w_0)_J\mid J \text{ $\mu$-orbit in $\Pia_\s$}\}.$$ 
Suppose that $w\in\Wa(\mu)$ and that $w\ne 1$. Then there exists a simple root $\a_j$ such that $w(\a_j)\in-\Dap$. We claim that if $J$ is the $\mu$-orbit to which $\a_j$ belongs, then $\D_J$ is of finite type.
If not, we would have that $w(\D_J\cap\Dap)\subset -\Dap$, which is impossible for $\D_J\cap\Dap$ is infinite. But then we can argue as in Proposition 9.17 of \cite{Carter} by induction on $\ell(w)$ to conclude.
\end{proof}
 \vskip5pt

\begin{cor}\label{fund} Let $\l\in\ha_\R^*$. Assume that $(\l,\a)\ge 0$  for 
any $\a\in\Dap$ and that $(\l+\rhat_\si)_{|\h_\p}=0$. Suppose that there is 
$w\in\Wa_\si$ such that $w(\l+\rhat_\si)_{|\h_\p}=0$.
Then  $w \in \Wa(\mu)$, hence $\tilde w\in
\widehat{W}_{\comm}$. 
\end{cor}
\begin{proof} 
We shall prove that if  $h\in\ha$ is  regular w.r.t. $\Da$, fixed by $\mu$   and such that $Re(\d(h))>0$, then  
\begin{equation}\label{imply}w\in\Wa_\s,\quad \mu(w(h))=w(h)\implies  w\in\Wa(\mu).\end{equation}
Since by  \eqref{coroot}Ê $h=\l+\rhat_\si$ is regular and $(\d,\l+\rhat_\s)>0$, the  claim follows from \eqref{imply} and  Proposition
\ref{Wmu}.
\par
To prove \eqref{imply}, remark that, since $\s\mu=\mu\s$, we have that $\mu(\Da)=\Da$. Since $\mu$ preserves the form $(\cdot,\cdot)$, we have that $\mu\Wa_\s\mu^{-1}=\Wa_\s$. Hence $w(h)=\mu(w(h))=
\mu w\mu^{-1}(\mu(h))=\mu w\mu^{-1} (h)$. Since $h$ is regular, then
$w=\mu w\mu^{-1}$.
 \end{proof}

\begin{cor}\label{commute}
Write $\s=\eta e^{2 \pi i ad(h)}$ with $h\in\h_0$ and assume that $\mu\eta=\eta\mu$ and that $\Pi_\eta$ is $\mu$-stable. Then we can choose $h$ and the element $w\in \Wa_\eta$ given by Proposition \ref{Pi} in such a way that $\mu(h)=h$ and $w\mu=\mu w$.
\end{cor}
\begin{proof}Write explicitly
$\Pi_\eta=\{\a_1,\ldots,\a_l\}$.
 Since $\s\mu=\mu\s$ and $\mu\eta=\eta\mu$, we have that $e^{2\pi i ad(h)}=e^{2\pi i ad(\mu(h))}$, hence
\begin{equation}\label{muh}\mu(h)-h\in \{x\in \h_0\mid \a_i(x)\in\ganz,\,1\leq i\leq l\}.\end{equation}
Choose $h$ such that $0\leq \a_i(h)< 1,\,1\leq i\leq l$. Since $\mu$ permutes $\Pi_\eta$, we have 
 $0\leq \a_i(\mu(h))< 1,\,1\leq i\leq l$, hence, by \eqref{muh}, $\mu(h)=h$.
 \par
 Let  $w\in\Wa_\eta$ be the element given by Proposition \ref{Pi}. The map from $\ha^*$ to $\ha^*$ given by $\a\mapsto \a-\a(h)\d$ maps $\Pia_\s$ to $w^{-1}(\Pia_\eta)$. Since $\Pia_\s$ is $\mu$-stable by construction and $\mu(h)=h$ we see that $w^{-1}(\Pia_\eta)$ is $\mu$-stable. In particular $w^{-1}(\rhat_\eta)_{|\h_\p}=0$. Since we are assuming that $\Pia_\eta$ is $\mu$-stable, we also have that $(\rhat_\eta)_{|\h_\p}=0$ , hence, by Corollary \ref{fund}, $w\mu=\mu w$.
\end{proof}
We now prove that 
$\widehat{W}_{\comm}$ is a Coxeter group. This is done using the abstract 
approach of Steinberg \cite{Steinberg}.  

\begin{lemma}\label{Sigma} Let $P=\{\a_J\mid \D_J\text{ of finite type}\}$, $
\Sigma=
\widehat{W}_{\comm}(P)$. Then
\begin{enumerate} 
\item $s_\a(\Sigma)=\Sigma$ for all $\a\in P$.
\item The elements of $\Sigma$ are nonisotropic.
\item The elements of $P$ are linearly independent.
\item $\Sigma=\Sigma^+\cup \Sigma^-,$ where each element of 
$\Sigma^+$ (resp. $\Sigma^-$) is a linear integral combination with positive (resp. negative) coefficients of elements of $P$.
\item If $\a\in\Sigma$ then $-\a\in\Sigma$ but no other multiple of $\a$ belongs to $\Sigma$.
\end{enumerate}
\end{lemma}
 \begin{proof}First of all observe that, if there is an orbit $J$ such that $\D_J$ is not of finite type, then, this orbit contains a component of $\Pia_\s$. By our assumption on $\mu$ this implies that there is only one orbit, thus, in this case, there is nothing to prove.
 
We can therefore assume that $\D_J$ is of finite type for any orbit. In this case:

1. This is an obvious consequence of the definition of $\Sigma$.
\par 2. It suffices to show that the elements of $P$ are nonisotropic and this is clear since the invariant form restricted to $\D_J$ is positive definite being $\D_J$ of finite type.
\par 3.
This follows from Lemma \ref{independent}.
\par 4. If $\a\in\Sigma$, then, by Lemma \ref{Wmu}, there is $u\in\Wa(\mu)$ 
and $\be\in\Pia_\s$ such that $\a=\tilde u((\be)_{|\ha^\mu})=
u(\be)_{|\ha^\mu}$. Since $u(\be)$ is a root for $\widehat L(\g,\si)$, then 
$u(\be)$ is an  integral linear combination of the elements of $\Pia_\s$ such 
that all the coefficients have the same sign. Restricting to $\ha^\mu$ we 
obtain the result.
\par 5. Suppose that $\a=w(\beta)$ with $w\in \Wa_{\comm}$ and $\beta\in P$. Then $-\a=ws_\beta(\beta)$. 
 Suppose now that $\a$ and $c\a$ are in $\Sigma$ with $c\ne\pm 1$. We can 
clearly assume that $\a\in P$.   Thus we can write $c\a=w(\beta)$ with 
$\a,\beta\in P$.  It is clear that $\be\ne\a$. By 3. and 4. above we can 
assume that $c\in\nat$, $c>1$.  Then $w^{-1}(\a)=\frac{1}{c}\beta$. By 4. this 
says that $\frac{1}{c}\beta=\sum_{\gamma\in P}x_\gamma\gamma$ with 
$x_\gamma\in\ganz$ and this contradicts 3. above.   
 \end{proof}
 
\begin{cor}\label{coxeter}
$(\widehat{W}_{\comm}, P)$ is a Coxeter system.
\end{cor}
\begin{proof}

Properties 1-5 in the above Lemma allow us to apply Deodhar's  ``root system condition" (cf. \cite[\S 2]{deo}).
The hypothesis of part 2) of the Main Theorem of \cite{deo} follow at once 
from Lemma \ref{Sigma} and the definition of $\Sigma$, possibly with the 
exception of (ii) in Deodhar's statement. This last condition follows from 4. 
and 5.
\end{proof}

\vskip5pt
Let $\Wa_\aa$ be the Weyl group of $\widehat L(\aa,\si)$. We need to realize $\Wa_\aa$ as a subgroup of $\widehat{W}_{\comm}$. This is accomplished by the next Lemma.
\begin{lemma}\label{subgroup}Let $\a$ be a real root of $\Da$.
Then $\a_{|\ha^\mu}$ is either isotropic or is a multiple of a root in $\Sigma$. In particular
  $s_{\a_{|\ha^\mu}}\in 
\widehat{W}_{\comm}$. 
\end{lemma}
\begin{proof}
First remark that the $\a_J$ are simple roots for an affine root system $\Sigma^{aff}$, whose set of real roots is $\Sigma$.
Take $\a\in\Dap$. Then $\a_{|\ha^\mu}\in\sum_J\ganz_{\geq 0} \a_J$, so we have a notion of height $ht_{r}$ ($r$ stands for ``restricted") for $\a_{|\ha^\mu}$.\par
Set $\beta=\a_{|\ha^\mu}$.We shall prove by induction on $ht_{r}$ that $\beta$ is either isotropic or is a multiple of a root in $\Sigma$. If $ht_r(\beta)=0$ then clearly $\beta$ is isotropic. So we can assume that 
$ht_r(\beta)>0$. Two cases arise.\par\vskip5pt
{\it 1. $(\beta,\a_J)\leq 0$ for all $J$.}
\vskip5pt\noindent
We prove then that $\beta$ is isotropic, by showing that it is an imaginary root of $\Sigma^{aff}$. We use Lemma 5.3 of \cite{Kac}, where it is shown (in our context) that the set 
$$\left\{\gamma \in \left(\sum_J\ganz_{\geq 0} \a_J\right)\setminus\{0\}:\, (\gamma,\a_J)\leq 0\,\forall\,J\text{ and }
Supp(\gamma)\text{ is connected }\right\}$$
consists of positive imaginary roots. We have only to verify that $Supp(\beta)$ is connected.
Suppose this is not the case. Then we can write $\beta=\beta_1+\beta_2$ with 
$Supp(\beta_1)=B_1,\,Supp(\beta_2)=B_2,\,
(\a_{I},\a_{J})=0$ for $I\in B_1,\,J\in B_2$. Since
 $$0=(\a_{I},\a_{J})=\frac{1}{|I| |J|}\sum_{i\in I, j\in J} (\a_i,\a_j)$$
and $J\cap I=\emptyset$, we have $(\a_i,\a_j)\leq 0\ \forall\,i\in I,
j\in J$, hence $(\a_i,\a_j)= 0\ \forall\,i\in I,
j\in J$. In turn we get that $Supp(\a)$ is not connected, and this is a contradiction (see e.g. \cite[Proposition 1.6]{Kac}).
\vskip5pt
{\it 2. There exists $I$ such that  $(\beta,\a_I)>0$.}
\vskip5pt\noindent
We have 
$$((w_0)_I(\a))_{|\ha^\mu}=s_{\a_I}(\beta)=\beta-c\a_I.$$
Since $(\beta,\a_I)>0$, we have that $c>0$, and in turn  $ht_r(\beta-c\a_I)<ht_r(\beta)$. If $\beta-c\a_I\in \sum_J\ganz_{\geq 0} \a_J$ we are done by induction. Otherwise  $s_{\a_I}(\beta)\in - \sum_J\ganz_{\geq 0} \a_J$  and this implies $(w_0)_I(\a)\in -\Dap$. Hence $\a\in\sum_{i\in I}\ganz \a_i$ and in turn it restricts to a multiple of $\a_I$, as wished.
\end{proof}

If $\beta\in \Da_{|\ha^\mu}$ and
 it is not isotropic, then we set 
$Red(\beta)$ to be the unique  element of $\Sigma$ such that 
$\beta=cRed(\beta)$ with $c>0$.
 By Lemma \ref{subgroup}, $\Wa_\aa$ is a reflection subgroup of the 
Coxeter group $\widehat{W}_{\comm}$ generated by the reflections in 
$Red(\Pia_\aa)$. Let $\Wa'$ be the set of minimal right  coset representatives 
of $\Wa_\aa$ in $\widehat{W}_{\comm}$. Recall from \cite{dyer} that $\Wa'$ can 
be characterized as
 \begin{equation}\label{mcr}
 \Wa'=\{w\in\widehat{W}_{\comm} \mid w^{-1}(Red(\Da^{re,+}_\aa))\subset
\Sigma^+\},
 \end{equation}
 where $\Da^{re,+}_\aa$ is the set of positive real roots in $\Dap_\aa$.
 \section{Main result}
 
 We  first recall (e.g. from \cite{KMP}) the construction of the Clifford module
 $F^{\si}(\p)$.  Suppose that $V$ is a complex finite dimensional vector space 
endowed with a simmetric bilinear form $(\cdot,\cdot)$ and $\s$ is an 
elliptic linear operator on $V$ (i.e. diagonalisable with modulus one 
eigenvalues), leaving the form invariant. Write 
$V=\bigoplus\limits_{\ov j\in\R/\ganz} V^{\ov j}$, where $V^{\ov j}$ is, 
as usual, the $\s$-eigenspace with the eigenvalue $e^{2\pi i j}$. Consider 
\begin{equation}\label{vbar}\ov V=\bigoplus_{j\in\R}(t^{j-\half}\otimes V^{\ov j})\end{equation}
endowed with the bilinear form $<t^i\otimes 
a,t^j\otimes  b>=\d_{i+j,- 1}(a,b)$. Let
$Cl(\ov V)$ be the corresponding Clifford algebra. We choose   a 
maximal  isotropic subspace
$\ov V^+$ of $\ov V$ as follows:   fix a  maximal isotropic subspace $U$  of  $V^{\ov 0}$, and let 
$$\ov V^+= (\bigoplus_{j>-\frac{1}{2}}(t^{j-\frac{1}{2}}\otimes V^{\ov j}
))\oplus(t^{-
\frac{1}{2}}\otimes U ).$$ \par
 Let
\begin{equation}\label{clifford}F^{\si}(V)=Cl(\ov V)/Cl(\ov V)\ov V^+.\end{equation} 
We specialize to the case when $V=\p$, the orthocomplement of $\aa$ in $\g$, and $\s$ is an elliptic automorphism of $\g$ restricted to $\p$. We then make the following specific choice for $U$:
let $\D(\p^{\ov j})$ be the set of $\h_0^\mu$-weights of $\p^{\ov j}$.
 Set
$\Dp(\p^{\ov 0})=\D(\p^{\ov 0})\cap( \D_0^+)_{|\h_0^\mu}$ and
$$\p^\pm=\sum\limits_{\a\in\pm\Dp(\p^{\ov 0})}\p^{\ov 0}_\a.$$ Thus we can write $$
\p^{\ov 0}=\h_\p\oplus\p^+\oplus\p^-. $$
Choose a maximal isotropic subspace $ \h_\p^+$ of $\h_\p$. Then
$U=\h_\p^+\oplus\p^+$.
Set $L=\dim\h_\p$
and $l=\lfloor\frac{L}{2}\rfloor=\dim\h_\p^+$. Fix a basis $\{v_i\}$ of
$\mathfrak \h_\p$ such that $\{v_i\mid i\le l\}$ is a basis of $\mathfrak \h_\p^+$ and
$(v_i,v_{L-j+1})=\d_{ij}$. 
Fix  weight vectors $X_\a\in\p^{\ov j}_\a$ and set, for $j\in\ov j,$ 
$i\in\ganz$ and $s=1,\ldots,L$, $$ \xi_{j,\a}=t^{j-\half}\t
X_\a,\qquad v_{i,s}=t^{i-\half}\t v_s. $$ Set also  
 \begin{align*} J_-=& \{(j,\a)\mid
j<0,\ \a\in\D(\p^{\ov j})\}\cup\{(0,\a)\mid \a\in-\Dp(\p^{\ov 0})\}\\
&\cup\{(i,s)\mid i\in\ganz, i<0,\ s=1,\dots,L\}\cup\{(0,s)\mid L-s+1\le l\}.
\end{align*} Putting any total order on $J_-$, the
pure spinors \begin{equation}\label{vector} v_{i_1,s_1}\dots
v_{i_h,s_h}\xi_{m_1,\be_1}\dots\xi_{m_k,\be_k} \end{equation} with
$(i_1,s_1)<\dots<(i_h,s_h)$ and $(m_1,\be_1)<\dots<(m_k,\be_k)$ in $J_-$
form a basis of $F^{\si}(\p)$.

It is shown in \cite[Lemma 5.1]{KMP} that the vector in \eqref{vector} is a weight vector for the  action of $\ha_\aa$ on $F^{\si}(\p)$ having weight 
\begin{equation}\label{weight}
\rhat_\si-\rhat_{\aa,\s}+\sum_qi_q\d+\sum_p 
(m_p\d+\be_{p}).
\end{equation}
Set
$$\mathcal S=\{\xi\in\ha^*\mid \xi=\sum_{\beta\in\Dap}n_\beta\beta,\,0\leq n_\beta\leq \dim\widehat L(\g,\sigma)_\beta,\,n_\beta=0\,\,\, a.e.\}.$$
It is easy to see (cf. \cite[Lemma 3.2.3]{Kumar}) that $\rhat_\s-\mathcal S$ is stable under the action of $\Wa_\s$.

In the following lemma we adapt the proof of Lemma 3.2.4 of \cite{Kumar} to the present situation.
\begin{lemma}\label{lemmaKumar}Assume
  $(\L+\rhat_\si)_{|\ha^\mu}=0$. Suppose that $\l\in (\ha^\mu)^*$ is of the form $\tau_{|\ha^\mu}+\nu$, with $\tau$ a weight of $L(\L)$ and $\varphi_\aa^*(\nu)-\rhat_{\aa\,\s}$ a $\ha_\aa$-weight of $F^{\si}(\p)$. Then there is $v\in\widehat{W}_\comm$ such that
$v(\l)=\L+\rhat_\s$ if and only if $v(\tau)=\L$ and $v(\nu)=\rhat_\s$.
\end{lemma} 
\begin{proof}Clearly
$$
v(\l)=\L+\rhat_\s\Leftrightarrow v(\tau_{|\ha^\mu})-\L=\rhat_\s-v(\nu).
$$
Since $\varphi_\aa^*(\nu)-\rhat_{\aa\,\s}$ is a weight of $F^{\si}(\p)$ then $\nu$ is a $\ha^\mu$-weight of  $F^{\si}(\p)\otimes 1\subset F^{\si}(\g)$. The $\ha^\mu$-weights of   $F^{\si}(\g)$ are in the set $\rhat_\s-\mathcal S_{|\ha^\mu}$. Since $v\in\widehat{W}_\comm$ and $\rhat_\s-\mathcal S$ is $\Wa_\s$-stable, we see that $v(\nu)\in \rhat_\s-\mathcal S_{|\ha^\mu}$. It follows in particular  that
$\rhat_\s-v(\nu)\in\sum_{\a\in P}\nat \a$. Since $v(\tau)$ is a weight of $L(\L)$, $\L_{|\h_\p}=0$ and $v\in\widehat{W}_\comm$, we have that
$(\L-v(\tau))_{|\ha^\mu}=-(v(\tau_{|\ha^\mu})-\L)\in\sum_{\a\in P}\nat\a$. Comparing these two observations we find that $(\L-v(\tau))_{|\ha^\mu}=0$ and $v(\nu)=\rhat_\s$. Since $\L-v(\tau)$ is a sum of positive roots, we find that $v(\tau)=\L$.

\end{proof}

For $w\in\Wa_\s$ set
$$N(w)=\{\a\in\Dap\mid w^{-1}(\a)\in -\Dap\}.$$
To simplify notation we set 
$\Vert \lambda\Vert=\Vert (\varphi_\aa^*)^{-1}(\l)\Vert$ 
whenever $\l\in \varphi_\aa^*((\ha^{\mu})^*)$. In the following proposition we need to exploit the assumption we made that the form $(\cdot,\cdot)$ is positive definite when restricted to the real space $(\h_0)^*_\R$ defined in \eqref{realh}. Moreover observe that, since $\mu$ stabilizes $\h_0$, it permutes the set $\D$ of $\h_0$-weights of $\g$, hence $\mu((\h_0)_\R)=(\h_0)_\R$. In particular we have the orthogonal decomposition
$$
(\h_0)^*_\R=(\h_0)^*_\R\cap(\h^\mu)^*\oplus(\h_0)^*_\R\cap \h^*_\p.
$$

Recall that a weight $\L\in\ha^*$ is said to be dominant if $(\L,\a)\in\R$ for any $\a\in\Da$ and $(\L,\a)\ge0$ for $\a\in\Dap$. If we write $\L=k\L_0+\L_{|\h_0}+(\L_0,\L)\d$ then $\L$ dominant implies 
$k=(\L,\d)\ge0$.

It is shown in \cite[Ch. 10, \S{} 5]{Helgason1} that $\D$ generates $\h_0$ over $\C$. This implies that $\l\in (\h_0)^*_\R$ if and only if $(\l,\a)\in\R$ for any $\a\in\D$. In particular if $\L\in\ha^*$ is such that $(\L,\a)\in\R$ for any $\a\in\Da$, then $\L_{|\h_0}\in(\h_0)^*_{\R}$. Thus we have an orthogonal decomposition
$$
\L_{|\h_0}=\L_{|\h_0^\mu}+\L_{|\h_\p}
$$
with 
\begin{equation}
\L_{|\h_0^\mu}\in(\h_0)^*_\R\cap(\h_0^\mu)^*\quad\L_{|\h_\p}\in(\h_0)^*_\R\cap \h^*_\p.
\end{equation}

Recall that $\L\in\ha^*$ is said to be integral if $2\frac{(\L,\a)}{(\a,\a)}\in\ganz$ for any simple root $\a$.
\begin{prop}\label{newkumar} 
 Suppose that $\Lambda\in\ha^*$ is dominant integral. Let  $\nu$ be a weight of $L(\L)\otimes F^{\s}(\p)$ such that 
$$
\Vert \nu+\rhat_{\aa\s}\Vert =\Vert \L+\rhat_\s\Vert.
$$
Then there is $w\in\What_\s$ such that 
\begin{equation}\label{wished}
w(\L+\rhat_\s)=(\varphi^*_\aa)^{-1}(\nu+\rhat_{\aa\s}).
\end{equation}
\end{prop}
\begin{proof}
Observe that $(\varphi_\aa^*)^{-1}(\nu + \rhat_{\aa\s})$ is a $\ha^\mu$-weight of $L(\L)\otimes F^{\s}(\p)\otimes F^{\s}(\aa)=L(\L)\otimes F^\s(\g)$, thus
$$
(\varphi_\aa^*)^{-1}(\nu+\rhat_{\aa\s})= (\l+\rhat_\s-s)_{|\ha^\mu}
$$
with $\l$ a weight of $L(\L)$ and $s\in \mathcal S$. Since $(\l+\rhat_\s-s)(K)=k+g>0$, we can find $v\in\What_\s$ such that $v(\l+\rhat_\s-s)$ is dominant. The set of weights of $L(\L)$ is $\What_\s$-invariant
and the same holds for $\rhat_\s-\mathcal S$,
hence we can write $v(\l+\rhat_\s-s)=\l'+\rhat_\s-s'$. 
It follows that $\Vert \l+\rhat_\s-s\Vert=\Vert\l'+\rhat_\s-s'\Vert$, so we have
$$
\Vert \L+\rhat_\s\Vert^2-\Vert \l+\rhat_\s-s\Vert^2=(\L+\rhat_\s+\l'+\rhat_\s-s',\L-\l'+s').
$$
Since $\L-\l'+s'$ is a sum of positive roots and $\L+\rhat_\s$, $\l'+\rhat_\s-s'$ are both dominant, we obtain that
$$
\Vert \L+\rhat_\s\Vert\ge\Vert \l+\rhat_\s-s\Vert.
$$
On the other hand
$$\Vert \nu+\rhat_{\aa\s}\Vert=\Vert (\l+\rhat_\s-s)_{|\ha^\mu}\Vert\le\Vert  \l+\rhat_\s-s\Vert \le \Vert\L+\rhat_\s\Vert$$
so, since $\Vert \nu+\rhat_{\aa\s}\Vert =\Vert \L+\rhat_\s\Vert$, we obtain equalities. Since $\L+\rhat_\s$ is regular we find that
$$
0=\Vert \L+\rhat_\s\Vert^2-\Vert \l+\rhat_\s-s\Vert^2=(\L+\rhat_\s+\l'+\rhat_\s-s',\L-\l'+s')$$
implies $\L=\l'$ and $s'=0$, so $\L+\rhat_\s=v(\l+\rhat_\s-s)$. Moreover $\Vert (\l+\rhat_\s-s)_{|\ha^\mu}\Vert=\Vert  \l+\rhat_\s-s\Vert$ implies $\Vert { (\l+\rhat_\s-s)}_{|\h_\p}\Vert=0$. Since $( \l+\rhat_\s-s)_{|\h_0}\in(\h_0)^*_\R$ and the form $(\cdot,\cdot)$ is positive definite on $(\h_0)^*_\R$ we obtain that ${ (\l+\rhat_\s-s)}_{|\h_\p}=0$ and
$(\l+\rhat_\s-s)_{|\ha^\mu}=\l+\rhat_\s-s$. Thus 
$$
(\varphi_\aa^*)^{-1}(\nu+\rhat_{\aa\s})=v^{-1}(\L+\rhat_\s).
$$
Taking $w=v^{-1}$ we obtain \eqref{wished}, as wished.
\end{proof}
We are now ready to  prove our  main theorem. 
\begin{proof}[Proof of Theorem  \ref{multiplet}]Ê By Proposition \ref{newkumar},  there exists $w\in\Wa_\s$ such that 
$w(\L+\rhat_\s)=(\varphi_\aa^*)^{-1}(\nu+\rhat_{\aa\,\s})$. By Corollary \ref{fund} we may assume that $w\in\Wa(\mu)$.
We claim that 
$\tilde w\in \Wa'$ and
that for any 
$\tilde w\in  \Wa'$ the corresponding submodule occurs with the 
prescribed multiplicity. The proof of the first statement follows from \eqref{mcr}.
For the second statement we first observe that, if $\tilde w\in \Wa'$ and $N(w)=\{\psi_1,\ldots,\psi_k\}$ with $\psi_i=n_i\d+\gamma_i$ then $\g_{\gamma_i}\subset\p$. Indeed any $Z\in \g_{\gamma_i}$ decomposes as $Z_\aa+Z_\p$ according to the orthogonal decomposition $\g=\aa\oplus\p$. If $Z_\aa\ne 0$ then $(\psi_i)_{|\ha^\mu}$ is a root of $\widehat L(\aa,\si)$. Since $w\in\Wa(\mu)$, $w((\psi_i)_{|\ha^\mu})=(w(\psi_i))_{|\ha^\mu}\in(-\Dap)_{|\ha^\mu}$. This implies that $Red((\psi_i)_{|\ha^\mu})\in N(\tilde w)$. This is not possible because $N(\tilde w)\subset \Sigma^+\backslash Red(\Da^{re,+}_\aa)$.
\par
Now consider the $2^{\lfloor \frac{\rank(\g^{\ov 0}))-\rank(\aa^{\ov 0})+1}{2}\rfloor}$ vectors $x\otimes y_i$, where $x$ is a weight vector in $L(\L)$ of weight $w(\L)$ and  $\{y_i\}=\{v_{0,j_1}\dots
v_{0,j_h}\xi_{-n_1,-\gamma_1}\dots\xi_{-n_k,-\gamma_k}\mid L-j_r+1\leq l\}$ 
(notation  as in \eqref{vector}).
Recalling that $\sum_{i=1}^k \psi_i=\rhat_\s-w(\rhat_\s)$, we see that all these vectors have weight $\nu$. 
By Lemma \ref{lemmaKumar}, if a vector in $X$ has weight $\nu$ then it is a 
linear combination of vectors $x\otimes y_i$. We end the proof by showing that 
they are highest weight vectors. If not, there exists a simple root 
$\a_i\in\Pia_\aa$ such that $\nu+\a_i$ is
a weight of $X$, and so 
\begin{equation}\label{contr}||\nu+\rhat_{\aa\,\s}+\a_i||^2=||\L+\rhat_\s||^2+||\a_i||^2+2(w(\L+\rhat_\s),\a_i)>||\L+\rhat_\s||^2.\end{equation}
(Note that $(w(\L+\rhat_\s),\a_i)>0$ since $\tilde w\in \Wa'$). On the other hand, by complete reducibility, a weight vector $v$ of weight $\nu+\a_i$ should belong to an irreducible highest weight $\widehat L(\aa,\si)$-module of highest weight $\eta$. 
It is a general fact that $||\nu+\rhat_{\aa\,\s}+\a_i||^2\leq ||\eta+\rhat_{\aa\,\s}||^2$, hence, by  \eqref{contr}, we have
$$
||\eta+\rhat_{\aa\,\s}||^2>||\L+\rhat_\s||^2.
$$
 Since $v\in Ker(D)$, this relation contradicts \eqref{azionedirac}.\end{proof}
\begin{rem}\label{land} If $\rank(\g^{\ov 0})=\rank(\aa^{\ov 0})$, formula \eqref{kerd}
specializes to formula (5.5) in Theorem 5.4 of \cite{KMP}. This latter
theorem is a generalization to arbitrary $\si$ of \cite[Theorem 16]{land}.
\end{rem}
\section{Decomposition of Clifford modules as representations of  orthogonal affine algebras}

 Given any complex finite dimensional vector space $V$, a nondegenerate symmetric bilinear form $(\cdot,\cdot)$ on $V$, and an elliptic automorphism $T$ of $V$,leaving $(\cdot,\cdot)$ invariant, we can construct the Clifford modules  $F^{T}(V)$. 
 
 In this section we use Theorem \ref{multiplet} to describe the decomposition of $F^{T}(V)$ as a $\widehat L(so(V),Ad(T))$-module. This is accomplished by considering the symmetric pair $(so(n+1), so(n))$.

We now describe this in full detail.
 Set $\tilde V = V 
\oplus \C$ and extend $(\ ,\ )$ to $\tilde V$ by setting 
$(v,1)=0,\,(1,1)=1$. Then
$so(V)$ embeds in $so(\tilde V)$. We endow $so(\tilde V)$ with the invariant 
form
$\langle X,Y\rangle =\half tr(XY)$.

 Extend $T$ to an automorphism $\tilde T$ of $\tilde V$ by setting $\tilde T(1)=1$.
For $v\in V$ define $X_v\in so(\tilde V)$ 
by $X_v(w+c)=cv-(v,w)$. 
Then $T X_v T^{-1}=X_{T(v)}$. 
Set $\s=Ad(\tilde T)$. 
Set also $\mu=Ad(\widetilde{-I_V})$ and note that $\mu \s=\s\mu$. Observe that $\mu$ is an involution of $so(\tilde V)$ and, if $so(\tilde V)=\aa\oplus\p$ is the corresponding eigenspace decomposition, then $\aa=so(V)$ and $\p=\{X_v\mid v\in V\}$. In particular the pair $(so(\tilde V),so(V))$ is a symmetric pair.
Note that, identifying $V$ with $\p$, the adjoint action of $so(V)$ on $\p$ 
gets identified with the natural action of $so(V)$ on $V$. 
  
Since  $F^{T}( V)$ is precisely $F^\s(\p)$, by applying our machinery we can turn it
 into a $\widehat L(so(V),\s)$-module. We wish to 
compute its decomposition into irreducible factors. 
In order to accomplish this, we observe that, by the explicit formula for the Dirac operator $D$ given in \cite[Lemma 4.5]{KMP}, $D$ acts trivially on $F^{\s}( \p)$, hence Theorem \ref{multiplet} provides the desired decomposition.
\par
\vskip 10 pt
Recall that $Ad(T)$ is an automorphism of $so(V)$ that is not of inner type if and only if $\dim V$ is even and $det(T)=-1$. Recall also that $det(T)=det(\tilde T)$. The $\widehat L(so(V),\s)$-structure of $F^T(V)$ depends on the type of $\s$ and of $\s_{|\aa}$: we now discuss the various cases.
\vskip10pt
Suppose first that $\dim V$ is even and that $det(T)=1$, so $\s_{|\aa}$ is of inner type, hence there is a Cartan subalgebra $\h$ of $\aa$ fixed by $\s$. Since $\dim \tilde V$ is odd, a Cartan subalgebra of $\aa=so(V)$ is also a Cartan subalgebra of $so(\tilde V)$.  Thus, in this case, $\h=\h_0=\h_0^\mu$, hence $\ha=\ha^\mu$ (i.e., we are in an affine equal rank setting). 
In this case $\widehat{W}_\comm =\Wa_\s$. We need to compute the coset representatives of $\Wa_\aa$ in $\Wa_\s$. Write $\s$ as $e^{2\pi i ad(h)}$ with $h\in\h_0$. Let $\{\a_1,\dots,\a_l\}$ be the set of simple roots of $so(\tilde V)$, let $\Theta$ be the highest root of both  $so(V)$ and $so(\tilde V)$. Observe that the Weyl group of $\widehat L(so(V),I_V)$ has index two in  the Weyl group of  $\widehat L(so(\tilde V),I_{\tilde V})$ and $\{1,s_{\a_l}\}$ is the set of minimal length coset representatives. Choose $w$  as in Proposition \ref{Pi} and observe that the Weyl group of $\widehat L(so(\tilde V),I_{\tilde V})$ stabilizes the set of roots of $\widehat L(so(V),I_V)$. It follows that the map $\a\mapsto w^{-1}(\a)+w^{-1}(\a)(h)\d$ is a bijection between the roots of $\widehat L(so(\tilde V),I_{\tilde V})$ and the roots of $\widehat L(so(\tilde V),\sigma)$ that maps the roots of $\widehat L(so(V),I_V)$ onto the roots of $\widehat L(so(V),\s)$. This implies that $\Wa_\aa$ has index two in $\Wa_\s$ and, if $\be_l=w^{-1}(\a_l)+w^{-1}(\a_l)(h)\d$, then $\{1,s_{\be_l}\}$ is the set of minimal length coset representatives. This implies that 
$$
F^\s(\p)=F^T(V)=V(\rhat_\s-\rhat_{\aa,\s})+V(s_{\be_l}(\rhat_\s)-\rhat_{\aa,\s}).
$$
Since the simple roots of $\widehat L(so(V),I_V)$ are $$\{\a_0=\d-\Theta,\a_1,\dots,\a_{l-1},s_{\a_l}(\a_{l-1})\},$$ we have that the simple roots of $\widehat L(so(V),\s)$ are $\{\be_0,\dots,\be_{l-1},s_{\be_l}(\be_{l-1})\}$ where $\be_i=w^{-1}(\a_i)+w^{-1}(\a_i)(h)\d$.

Set $\tilde\L_i$ be the fundamental weights of $\widehat L(so(V),\s)$ normalized by setting $\tilde\L_i(d)=0$ and set $s_l=\be_l(d)$. It is clear that 
$(\rhat_\s-\rhat_{\aa,\s})(\be_i^\vee)=0$ for $i<l$ and $(\rhat_\s-\rhat_{\aa,\s})(s_{\be_l}(\be_{l-1}^\vee))=1$. Analogously $(s_{\be_l}(\rhat_\s)-\rhat_{\aa,\s})(\be_i^\vee)=0$ for $i< l-1$, $(s_{\be_l}(\rhat_\s)-\rhat_{\aa,\s})(s_{\be_l}(\be_{l-1}^\vee))=0$, and $(s_{\be_l}(\rhat_\s)-\rhat_{\aa,\s})(\be_{l-1}^\vee)=1$. This implies that 
\begin{equation}\label{dimVevendetT1}
F^\s(\p)=F^T(V)=V(\tilde\L_l)+V(\tilde \L_{l-1}-s_l\d).
\end{equation}
\vskip 10 pt

Suppose now that  $\dim V$ is odd and that $det(T)=1$. Then $\s$ is an inner automorphism of $so(\tilde V)$, hence $\h_0$ is a Cartan subalgebra of $so(\tilde V)$. Since $\s_{|\aa}$ is of inner type, we have that $\h_0^\mu$ is a Cartan subalgebra of $so(V)$.   This time $\h_0\ne \h_0^\mu$, thus we need to identify the group $\widehat{W}_\comm$. Since all the orbits of $\mu$ on the set of simple roots of $so(\tilde V)$ are made of orthogonal roots, the restriction of the set of roots of $so(\tilde V)$ to $\h_0^\mu$ is the set of roots of $so(V)$. In particular, since the highest root $\Theta$ of $so(\tilde V)$ is fixed by $\mu$, $\Theta_{|\h_0^\mu}$ is the highest root of $so(V)$. This implies that, if  $\Pi=\{\a_1,\dots,\a_l\}$ is the set of simple roots of $so(\tilde V)$ and $\Pia=\{\d-\Theta,\a_1,\dots,\a_l\}$, then  $\Pia_{|\ha^\mu}$ is the set of simple roots of $\widehat L(so(V), I_{so(V)})$. Let $\Wa$ be the Weyl group of $\widehat L(so(\tilde V), I_{\tilde V})$. Choose $h\in\h_0$ such that $\s=e^{2\pi i ad(h)}$ and let $w\in\Wa$ be the element given in  Proposition \ref{Pi}. Since $\h_0^\mu$ is a Cartan subalgebra of $so(V)$ and $\h_0$ is its centralizer in $so(\tilde V)$, we can choose $\Pi$ to be $\mu$-stable. We can therefore apply Corollary \ref{commute} to have $h\in \h_0^\mu$ and $w\mu=\mu w$. Hence $\tilde w=w_{|\ha^\mu}$ is an element of the Weyl group of $\widehat L(so(V),I_{so(V)})$ and the map $\a\mapsto w^{-1}(\a)+w^{-1}(\a)(h)\d$ restricts to $(\ha^\mu)^*$, mapping $\Pia_{|\ha^\mu}$ onto a set of simple roots of $\widehat L(so(V),\s)$. Moreover, this set is clearly $(\Pia_\s)_{|\ha^\mu}$. This implies that $\Wa_\comm$ is the Weyl group of $\widehat L(so(V),\s)$, so $\Wa'=\{1\}$ and 
$$
F^{\s}(\p)=F^T(V)=2V(\rhat_\s-\rhat_{\aa,\s}).
$$
Set $\be_i=w^{-1}(\a_i)+w^{-1}(\a_i)(h)\d$ for $i=1,\dots,l$ and $\be_0=w^{-1}(\d-\Theta)+w^{-1}(\d-\Theta)(h)\d$.
Assume that we labeled simple roots so that $\a_i=(\a_i)_{|\h_0^\mu}$ if $i<l-1$ and that $(\a_{l-1})_{|\h_0^\mu}=(\a_l)_{|\h_0^\mu}$ is the short simple root of $so(V)$. Then
$(\rhat_\s-\rhat_{\aa,\s})(\be_i^\vee)=0$ for $i<l-1$ while $(\rhat_\s-\rhat_{\aa,\s})((\be_{l-1}{}_{|\h_0^\mu})^\vee)=1$, thus
\begin{equation}\label{dimVodddetT1}
F^{\s}(\p)=F^T(V)=2V(\tilde \L_{l-1}).
\end{equation}

If $\dim V$ is odd and $det (T)=-1$, then $\s_{|\aa}$ is of inner type while $\s$ is not. This implies that $\h_0^\mu$ is a Cartan subalgebra of $so(V)$ and, since $rk(so(\tilde V))=rk(so(V))+1$, its centralizer $\h_0$ in $\g^{\ov 0}$ must be $\h_0^\mu$. Hence $\ha=\ha^\mu$ and $\Wa_\comm=\Wa_\s$. Write as usual $\s=\eta e^{2\pi i ad(h)}$ with $h\in\h_0$. Since $\mu(h)=h$ we have that $\mu\eta=\eta\mu$. By inspection one checks readily that this implies $\eta=\mu$, thus we can write $\s=\mu e^{2 \pi i ad(h)}$. Let $\Pi=\{\a_1,\dots,\a_l\}$ be the set of simple roots of $so(V)$ and $\theta$ the highest weight of $V$ as a $so(V)$-module, so that $\Pia=\{\half\d-\theta,\a_1,\dots,\a_l\}$ is a set of simple roots for $\widehat L(so(\tilde V),\mu)$. Assume that the roots are  labeled so that $(\theta,\a_1)=1$. Then, since $\theta=\sum_{i=1}^l \a_i$, we have that $s_\theta(\a_1)=-\Theta$, where $\Theta$ is the highest root of $so(V)$. This implies also that $\d-\Theta=s_{\half\d-\theta}(\a_1)$. It is known (see \cite{IMRN}) that index of the Weyl group of $\widehat L(so(V),I_V)$ in the Weyl group of $\widehat L(so(\tilde V),\mu)$ is two and that the set of minimal coset representatives is $\{1,s_{\half\d-\theta}\}$. Since a set of simple roots for $\widehat L(so(V),I_V)$ is $\{\d-\Theta,\a_1,\dots, \a_n\}=\{s_{\half\d-\theta}(\a_1),\a_1,\dots,\a_l\}$ we see that $\Wa_\mu$ stabilizes the roots of $\widehat L(so(V),I_V)$. Hence, arguing as in the previous cases, we can choose $w\in\Wa_\mu$ as in Proposition  \ref{Pi} and find that $\Wa_\aa$ has index two in $\Wa_\s$ and the set of minimal length coset representatives is $W'=\{1,s_{\be_0}\}$ where $\be_0=w^{-1}(\half\d-\theta)+w^{-1}(\half\d-\theta)(h)\d$. Thus Theorem \ref{multiplet} shows that
$$
F^\s(\p)=F^T(V)=V(\rhat_\s-\rhat_{\aa,\s})+V(s_{\beta_0}(\rhat_\s)-\rhat_{\aa,\s}).$$
Setting $\be_i=w^{-1}(\a_i)+w^{-1}(\a_i)(h)\d$, we find, arguing as above, that
\begin{equation}\label{dimVodddetT-1}
F^\s(\p)=F^T(V)=V(\tilde \L_0)+V(\tilde \L_1-s_0\d).
\end{equation}

In the last case we have that $det(T)=-1$ and $\dim V$ is even. In this case $\s_{|\aa}$ is not of inner type, so $\dim \h_0^\mu= \dim V-1$, while, since $\s$ is of inner type, $\dim \h_0=\dim V$. By Proposition \ref{Pi} we can write $\Pia_\s=\{s_0\d-\Theta,s_1\d+\a_1,\dots,s_l\d+\a_l\}$, where $\{\a_1,\dots,\a_l\}$ is a set of simple $\h_0$-roots for $so(\tilde V)$ and $\Theta$ is the corresponding highest root. We now prove that if $\a\in\Pia_\s$ then $\a_{|\ha^\mu}$ is a root of $\widehat L(so(V),\s)$. Since $\mu$ induces a nontrivial automorphism of the diagram of $\widehat L(so(\tilde V), \s)$ we see that $\mu$ exchanges $s_0\d-\Theta$ with $s_1\d+\a_1$ and fixes all the other simple roots. This implies that $s_0=s_1$ and that $\mu(\a_1)=-\Theta$. Let $X$ be a nonzero element of $so(\tilde V)^{\ov{s_1}}_{\a_1}\oplus so(\tilde V)^{\ov{s_0}}_{-\Theta}$ that is fixed by $\mu$. Then $t^{s_0}\otimes X$ is in $\widehat L(so(V),\s)$ and its $\ha^\mu$-weight is $(s_0\d-\Theta)_{|\ha^\mu}=(s_1\d+\a_1)_{|\ha^\mu}$. If $\a=s_i\d+\a_i$ is fixed by $\mu$ then we claim that $\mu(X_{\a_i})=X_{\a_i}$. Indeed, if $\mu(X_{\a_i})=-X_{\a_i}$, then $X_{\a_i}=X_v$ for some $v\in V$. If $h\in\h_\p$ then $h=X_w$ with $w\in V$ and, since $\mu(\a_i)=\a_i$, $\a_i(h)=0$. But then $[h,X_{\a_i}]=[X_w,X_v]=0$. One easily computes  that $[X_w,X_v](c+u)=(u,w)v-(u,v)w$, hence $[X_v,X_w]=0$ if and only if $v$ and $w$ are linearly dependent. In turn,  this implies $X_{\a_i}\in\h_\p$ which is absurd. It follows that $t^{s_i}\otimes X_{\a_i}$ is an element  of $ \widehat L(so(V),\s)$ having weight $(s_i\d+\a_i)_{|\ha^\mu}$, as desired.

Since in this case the orbits of $\mu$ on $\Pia_\s$ are made of orthogonal roots, the proof of Lemma \ref{subgroup} implies that $\Da_{|\ha^\mu}\subset \Sigma$. Having shown that $(\Pia_\s)_{|\ha^\mu}$ is a set of roots of $\widehat L(so(V),\s)$ we deduce that $\Pia_\aa=(\Pia_\s)_{|\ha^\mu}$. In particular $\Wa_\comm=\Wa_\aa$ and  $\widehat W'=\{1\}$.
Theorem \ref{multiplet} implies in this case that
$$
F^{\s}(\p)=F^T(V)=2V(\rhat_\s-\rhat_{\aa,\s}).
$$
Set $\be_{i-1}=(s_i\d+\a_i)_{|\ha^\mu}$ for $i=2,\dots,l$ and $\be_0=(s_0\d-\Theta)_{|\ha^\mu}=(s_1\d+\a_1)_{|\ha^\mu}$.
 Then
$(\rhat_\s-\rhat_{\aa,\s})(\be_i^\vee)=0$ for $i>0$ while, since $\Vert\be_0\Vert^2=\half\Vert\Theta\Vert^2$, $(\rhat_\s-\rhat_{\aa,\s})(\be_{0}^\vee)=1$, thus
\begin{equation}\label{dimVevendetT-1}
F^{\s}(\p)=F^T(V)=2V(\tilde \L_{0}).
\end{equation}
\vskip 10 pt
Let us
apply the above discussion to the special cases when $T=\pm I_V$.
Since $\s_{|\aa}=I_{so(V)}$ in these cases,
 $F^{\pm I_V}(V)$ is a $\widehat L(so(V),I_{so(V)})$-module. Let $\Pi=\{\a_1,\dots,\a_l\}$ be the set of simple roots for $so(V)$ labeled as in \cite[TABLE Fin]{Kac} and let $\Theta$ be the corresponding highest root. Setting $\a_0=\d-\Theta$, then $\Pia_\aa=\{\a_0,\dots,\a_l\}$. Let $\tilde \L_i$ be the corresponding fundamental weights. 
 
 If $T=-I_V$, it follows from Lemma \ref{eta}  (or rather from its proof) that, if $\theta$ is the highest weight of $V$, then $\Pia_\s=\{\half\d-\theta,\a_1,\dots, \a_l\}$, hence, since $s_{\frac{1}{2}\d-\theta}(\a_1)=\d-\Theta$, it follows from \eqref{dimVevendetT1} and \eqref{dimVodddetT-1} that
$$
F^{-I_V}(V)=V(\tilde\L_0)+V(\tilde\L_1-\half\d).
$$
(Note the different labeling of the roots in \eqref{dimVevendetT1}.)

If $T=I_V$ and $\dim V$ is even, we can choose a root $\be$ for $so(\tilde V)$ so that $\{\a_1,\dots,\a_{l-1},\be\}$ is a set of simple roots for  $so(\tilde V)$ and $\a_l=s_{\be}(\a_{l-1})$. Since in this case  $\Pia_\s=\{\d-\Theta,\a_1,\dots,\a_{l-1},\be\}$ we deduce from \eqref{dimVevendetT1} that
$$
F^{I_V}(V)=V(\tilde\L_{l-1})+V(\tilde\L_l).
$$
If $\dim V$ is odd then, since in this case $(\Pia_\s)_{|\ha^\mu}=\{\d-\Theta,\a_1,\dots,\a_l\}$, then  relation \eqref{dimVodddetT1} implies  
$$
F^{I_V}(V)=2V(\tilde\L_l).
$$
\vskip10pt
The previous discussion explains why  the Clifford modules $F^{-I_V}(V),$ $F^{I_V}(V)$ are also called the basic+vector  and the  spin representation of  \break $\widehat L(so(V),I_{so(V)})$, respectively.

\section{Decomposition rules of level $1$ modules for symmetric pairs.}
We now assume that $\mu$ is an indecomposable involution of $\g$  and write $\g=\k\oplus \p$ for the corresponding (complex) Cartan decomposition. In this section we apply Theorem \ref{multiplet} with $\L=0$ to the following two special cases. In the first case we take $\s=\mu$ and $\aa=\k$, while in the second case we take $\s=I_\g$ and $\aa=\k$. In the first case $F^{\s}(\p)=F^{-I_\p}(\p)$ thus it realizes the basic+vector representation of $\widehat L(so(\p),I_\p)$, while in the second case $F^{\s}(\p)=F^{I_\p}(\p)$ so it is its spin representation.
Since the pair $(\g,\k)$ is symmetric, it follows from the explicit expression for $D$ given in \cite[Lemma 4.5]{KMP} that $D$ acts trivially on $F^\s(\p)$. Since the action of $\widehat L(\k, I_\k)$ on $F^\s(\p)$ is just the restriction of the action of $\widehat L(so(\p),I_{so(\p)})$ to it, Theorem \ref{multiplet} provides the decomposition rules for the basic+vector and the spin representation of  $\widehat L(so(\p),I_{so(\p)})$ when restricted to $\widehat L(\k, I_\k)$. In this way we recover the formulas we already found in \cite{KMP}.

\subsection{Decomposition of basic+vector representations}\label{bv}

Since in this case $\s=\mu$ we have clearly $\h_0=\h_0^\mu$, so
$\widehat{W}_\comm =\Wa_\s$ and $$
\Wa'=\{w\in\Wa_\s\mid N(w)\subset \Dap\backslash \Dap_\k\}.
$$ 
 Since $\rhat_\s=g\L_0+\rho_0$ and $\rhat_{\k,\s}=\sum_ig_i\L_0^i+\rho_0$ we see that
 $$
 Ker D=F^\s(\p)=\sum_{w\in \Wa'}V(\sum_i(g-g_i)\L_0^i+\sum_{\a\in N(w)}\a).
 $$
 (See \cite[Theorem 3.5]{CKMP}).
 \subsection{Decomposition of spin representations}\label{sp}
We consider four cases:
\begin{enumerate}
\item $\g$ is simple and $\mu$ of inner type.
\item $\g$ is not simple.
\item $\g$ is simple of type $A_{2n+1}$, $D_{n}$, $E_6$ and $\mu$ not of inner type. 
\item $\g$ is simple  of type $A_{2n}$ and $\mu$ not of inner type.
\end{enumerate}
\subsubsection{Case 1.}
In this case $\ha=\ha^\mu$ thus $\widehat{W}_\comm=\Wa_{I_\g}$ and $$
\Wa'=\{w\in\Wa_{I_\g}\mid N(w)\subset \Dap\backslash \Dap_\k\}.
$$ 
 Let $\rho$, $\rho_{\k}$  be half the sum of the positive roots of  $\g$ and $\k$ respectively. Then $\rhat_{I_\g}=g\L_0+\rho$ and $\rhat_{\k,I_\k}=\sum_ig_i\L_0^i+\rho_\k$. It follows that
 $$
 Ker D=F^{I_\g}(\p)=\sum_{w\in \Wa'}V(\sum_i(g-g_i)\L_0^i+\rho-\rho_\k+\sum_{\a\in N(w)}\a).
 $$

\subsubsection{Case 2.}In this case $\g=\mathfrak{s}\oplus\mathfrak{s}$ is the sum of two copies of a simple algebra $\mathfrak{s}$,  $\s=I_\g$ and $\mu$ is the flip automorphism $\mu(X,Y)=(Y,X)$. It follows that
 $\k$ is the diagonal copy of $\mathfrak{s}$ in $\mathfrak{s}\oplus\mathfrak{s}$. If $\h_{\mathfrak{s}} $ is a Cartan subalgebra of $\mathfrak{s}$, then $\h=\h_{\mathfrak{s}} \oplus\h_{\mathfrak{s}} $  and $\h_0^\mu$ is the diagonal copy of $\h_{\mathfrak{s}} $ in $\h$. It follows that $P=\Pia_{|\ha^\mu}=\Pia_\k$, hence $\Wa_\k=\widehat{W}_\comm$.
Let $(\cdot,\cdot)_{\mathfrak{s}}$ be the form $(\cdot,\cdot)$ restricted to the first factor of $\g=\mathfrak{s}\oplus\mathfrak{s}$. Since the form is $\mu$-invariant, we see that $(\cdot,\cdot)=(\cdot,\cdot)_{\mathfrak{s}}\oplus (\cdot,\cdot)_{\mathfrak{s}}$. If $2g_{\mathfrak{s}}$ is the eigenvalue of the Casimir of $\mathfrak{s}$ when acting on $\mathfrak{s}$ then the eigenvalue of the Casimir of $\g$ when acting on $\g$ is $2g_{\mathfrak{s}}$. Hence $g=g_{\mathfrak{s}}$. On the other hand, identifying $\mathfrak{s}$ and $\k$, since $(\cdot,\cdot)_{|\k}=2(\cdot,\cdot)_{\mathfrak{s}}$, we see that the eigenvalue of the Casimir of $\k$ when acting on $\k$ is $g_{\mathfrak{s}}$. Letting $\rho$ be half the sum of the positive roots of $\mathfrak{s}$, we deduce that $\rhat_{I_\g}=g_{\mathfrak{s}}\L_0+2\rho$, while $\rhat_{\k,I_\k}=\frac{g_{\mathfrak{s}}}{2}\L_0+\rho$, i.e. $\rhat_{I_\g}=2\rhat_{\k,I_\k}$. Thus Theorem \ref{multiplet} (with $\L=0$) in this case gives that
$$
Ker D=F^{I_\g}(\p)=2^{\lfloor\frac{ rank \g+1}{2}\rfloor}V(\rhat_{\k,I_\k}).
$$
 
 \subsubsection{Case 3.}First note that $\h_0^\mu$ is a Cartan subalgebra of $\k$ (since in this case $\s=I_\g$). Write as usual $\mu=\eta e^{2\pi i ad(h)}$ with $\eta$ a diagram automorphism and $h\in\h_0^\mu$.  If $l$ is the rank of $\k$ set $\Pi_\eta=\{\be_1,\dots,\be_l\}$ and let $\o_i$ be the unique element of $\h_0^\mu$ such that $\be_i(\o_j)=\d_{i,j}$. Set also $\o_0=0$. Since $\mu^2=I_\g$ Kac's classification of finite order automorphisms  implies that there is $i$ such that $\mu=\eta e^{\pi i \o_i}$.
 Let $\D_{s}$ ($\D_{l}$) be the set of short (long) roots of $\Phi_\eta$. If $p=0,1$, set $\D_{l}^p=\{\beta\in \D_{l}\mid \beta(\o_i)\equiv p \mod 2\}$. Set also $\Da_{l}^p=\{m\d+\a\mid m\in\ganz,\ \a\in\D_{l}^p\}$. 
 
 Let $\Theta$ be the highest root of $\g$. Since the $\mu$-orbits
 in $\Pi$ are made of orthogonal roots, we have that
 $\Phi_{|\h_0^\mu}=\Phi_\eta$. Hence
 $\Theta=\Theta_{|\h_0^\mu}$  is the highest root of
 $\k_\eta$. This implies that $P$ is the set of simple roots of
 $\widehat L(\k_\eta,I_{\k_\eta})$. It follows that
 $\widehat{W}_\comm =\Wa_{\k_\eta}$. Moreover  $\Sigma$ is the set of real roots of $\widehat L(\k_\eta,I_{\k_\eta})$ i. e.  the roots of the form $m\d+\be$ with $\be\in\Phi_\eta$.
 
  It is shown in \cite[\S{} 4.4.2]{CKMP}
 that $\D_\k=\D_{s}\cup\D_{l}^0$ is the set of roots of $\k$. Hence the real roots of   $\widehat L(\k,I_\k)$ are the roots of the form  $m\d+\be$ with $\be\in\D_{s}\cup\D_{l}^0$ with $m\in\ganz$. It follows that the real roots of  $\Da_\aa$ are roots in $\Sigma$ and, if $\widehat \Phi_\eta$ is the set of roots of $\widehat L(\k_\eta,I_{\k_\eta})$ and $\widehat \Phi^+_\eta$ is the set of positive roots, then
  $$\Wa'=\{w\in\Wa_{\k_\eta}\mid N(w)\subset \widehat \Phi_\eta\backslash \Dap_{\k}\}=\{w\in\Wa_{\k_\eta}\mid N(w)\subset  \Da_{l}^1\cap\widehat \Phi^+_\eta\}.
  $$ 
  Set $\rhat=\rhat_{I_\g}$.
 If $\be$ is a simple root of $P$ then $\be=\a_{|\ha^\mu}$ with $\a\in\Pia_{I_\g}$. Using the fact that $\g$ is simply laced we find that $(\rhat,\be)=(\rhat,\a)=\frac{(\a,\a)}{2}$ is independent of $\a$.  Setting $a_0=\frac{(\a,\a)}{2}$, it follows that $\rhat=a_0\nu(\rhat')$. Here $\nu$ is the isomorphism from $\ha^\mu$ to $(\ha^\mu)^*$ induced by the form $(\cdot,\cdot)$ and $\rhat'$ is the unique element in $\C d\oplus\h_0^\mu$ such that $\be(\rhat')=1$ for all $\be\in P$. The final outcome in this case is that
 $$
Ker D=F^{I_\g}(\p)=2^{\lfloor\frac{ dim\h_\p+1}{2}\rfloor}\sum_{w\in \Wa'}V(\varphi_\aa^*(a_0w(\nu(\rhat')))-\rhat_\k).
$$
 
\subsubsection{Case 4.}If $\g$ is of type $A_{2n}$ then $\mu$ is a diagram automorphism of $\g$. Recall that we are setting $\s=I_\g$ so $\h_0^\mu$ is a Cartan subalgebra of $\k$. Let $\D_\k$ be the set of roots of $\k$ and $\Dp_\k$ a set of positive roots.

Let $\D_{s}$, $\D_{l}$ be as in the previous case.  
Set also $\Da_{s}=\{m\d+\a\mid \a\in\D_{s},\,m\in\ganz\}$.

Note that in this case $P$ coincides with the set given in \cite[(8.3.6)]{Kac} (and called $\Pi$ there). In \cite[\S\ 8.3]{Kac} it is shown that this set corresponds to a set of simple roots for the set 
$\{m\d+\a\mid \a\in\D_\k,\,m\in2\ganz\}\cup\{\a+m\d\mid \a\in\D_\p,\,m\in2\ganz+1\}\cup\{m\d\mid m\in\ganz,\ m\ne0\}$.
 It follows that 
 $$\Sigma=\{m\d+\a\mid \a\in\D_\k,\,m\in2\ganz\}\cup\{\a+m\d\mid \a\in\D_\p,\,m\in2\ganz+1\}.
 $$
 As shown in Section  4.2.2 of \cite{CKMP}, we have that, if $\Phi$ is the set of roots of $\g$, then $\Phi_{|\h_0^\mu}=\D_\k\cup2\D_{s}$ and, if $\D_\p$ is the set of nonzero weights in $\p$, then $\D_\p=\Phi_{|\h_0^\mu}$. The set of real roots of  $\widehat L(\k,I_\k)$ is $\Sigma_\k=\{\a+m\d\mid\a\in\D_\k,\,m\in\ganz\}$. It follows that $\Sigma_\k$  is a subset of $\Sigma$ and
\begin{align*}
\Wa'&=\{w\in\widehat{W}_\comm \mid N(w)\subset
\Sigma\backslash\Sigma_\k\}\\&=\{w\in\widehat{W}_\comm \mid N(w)\subset (\d+2\Da_{s})\}.
\end{align*}
Defining $\rhat$, $a_0$, $\rhat'$, and $\nu$ as in the previous case, we find in the same way that $\rhat=a_0\nu(\rhat')$.
 The final outcome in this case is that
 $$
Ker D=F^{I_\g}(\p)=2^{\lfloor\frac{ dim\h_\p+1}{2}\rfloor}\sum_{w\in \Wa'}V(\varphi_\aa^*(a_0w(\nu(\rhat')))-\rhat_\k).
$$
\begin{rem}
The description of $\Wa'$ given in all the above cases  quickly leads to a combinatorial interpretation of the highest weights occurring in the decomposition of $Ker D$ in terms of abelian subspaces of $\p$. We refer the interested reader to Sections 3.1, 4.1, 4.3.2 and 4.3.3 of \cite{CKMP}.
\end{rem}

\section{Asymptotic dimensions}
The asymptotic dimension of an  integrable irreducible highest weight module over an affine algebra is a positive
real number, which has all properties of the usual dimension (e.g.,  it is well-behaved under tensor products and finite direct sums). In this section we discuss some results
on the asymptotic dimension of multiplets, and we take the occasion to correct the proof of a result which has been (correctly) stated in \cite{KMP}.\par
Let $V=V(\L)$ be an integrable irreducible highest weight module with highest weight $\L$
over the affine algebra 
$\widehat L(\aa,\s)$, with $\aa$ 
semisimple Lie algebra and $\s$ indecomposable. 
The series
$$ch_V(\tau,h)=tr_V e^{2\pi i (-\tau d_\aa+h)}$$
converges to an analytic function of the complex variable $\tau$, if 
$Im\,\tau>0$, for each $h$ in a Cartan subalgebra of $\aa^{\ov 0}$.
The asymptotics of this function  is as
follows:
\begin{equation}\label{cinquenove}ch_V(it,ith)\approx 
a(\L)e^{\frac{\pi c(k)}{12t}},\end{equation}
as $t\in \R^+,\, t\to 0$.
Here $k=\L(K)$ is the level of $\L$, $c(k)$ is the conformal anomaly \cite[(12.8.10)]{Kac} and $a(\L)$ is a positive
real number independent of $h$ called the asymptotic dimension of $V(\L)$.\par
We want to extend the notion of asymptotic dimension to the reductive case. Hence, let now
$\aa$ be a reductive Lie algebra and let  $\aa=\bigoplus\limits_{j=0}^s
\aa_j$ be its decomposition   into the direct sum of the eigenspaces for the action of the Casimir of $\aa$. We assume that $\aa_0$ corresponds to the zero eigenvalue, i.e., $\aa_0$ is abelian.   For each $j$ we can write $\aa_j=\oplus_{i} \aa_{ji}$ for the decomposition of $\aa_j$ into $\s$-indecomposable ideals.

Since $V$ is irreducible, then it is 
an outer tensor product of irreducible $\widehat
L(\aa_{ji},\sigma)$-modules with highest weights $\L^{ji}$ of level 
$k_j$. We define the {\it 
asymptotic dimension} of $V$ as
$$\asdim(V)=\prod_{j=1}^s( \prod_i a(\L^{ji})),$$
and we set  $\asdim V=1$ if $\aa$ is abelian.\vskip5pt

If $\h_\aa$ is a Cartan subalgebra of $\aa^{\ov 0}$ then set 
$$\ha_{(0)}^*=(\aa_0^{\ov
0})^*\oplus\C\L^0_0,\quad\ha_{(1)}^*=(\h_\aa\cap
\sum_{j>0}\aa_j)^*\oplus\sum_{j>0}\C\L_0^j,$$
so that any $\l\in\ha_\aa^*$ can be uniquely written as
$\l=\l_{(0)}+\l_{(1)}+a\d_\aa,\,\l_{(0)}\in\ha_{(0)}^*,\l_{(1)}\in\ha_{(1)}^*,
a\in\C$.
Note that, by the above convention, 
\begin{equation}\label{dimensioneasintotica}
\asdim(V(\L))=\lim_{t\to 0^+}e^{-\frac{\pi 
}{12t}\sum\limits_{j=1}^s(\sum_ic_{ji}(k_j))}ch_{V(\L_{(1)})}(it,th).\end{equation}

\vskip5pt
Let us now return to the setting of the previous sections and assume furthermore that $\aa^{\ov 0}$ is an equal rank subalgebra of $\g^{\ov 0}$ so that $\h_0$ is a Cartan subalgebra of $\aa^{\ov 0}$.

On the algebra $Cl(\ov V)$ (see \eqref{vbar}) there is a unique  involutive 
automorphism such that  $x\mapsto-x$ for $x\in \ov V$.  Then, 
denoting by $Cl(\ov V)^\pm$ the $\pm1$ eigenspace  for 
this
automorphism, we can write $$Cl(\ov V)=Cl(\ov 
V)^+\oplus Cl(\ov V)^-.$$ 
It follows
that 
$$
F^{\s}(V)=F^{\s}(V)^+\oplus F^{\s}( V)^-,
$$
where
$F^{\s}(V)^\pm=Cl(\ov V)^\pm/(Cl(\ov 
V)\ov V^+\cap
Cl(\ov V)^\pm)$. In Section 5.2 of \cite{KMP} we proved that  $F^{\s}(V)^\pm$ 
are $Cl(\ov V)^+$-stable. Then  $F^{\s}(V)^\pm$ are $\widehat L(\aa,\s)$-
modules and, moreover, the so-called ``homogeneous 
Weyl-Kac character formula" holds:
\begin{equation}\label{hwk}
L(\L)\otimes F^{ \s}( \p)^+-L(\L)\otimes F^{ 
\s}(\p)^-=
\sum_{w\in \widehat 
W'}(-1)^{\ell(w)}V(\varphi^*_\aa(w(\L+\rhat_\si))-\rhat_{\aa\si}).
\end{equation}

We want formula \eqref{hwk} to make sense also when the representatives $w$ in the r.h.s. are not minimal. To accomplish this goal, 
if $\l\in\ha_\aa^*$  is dominant integral for $\Dap_\aa$ and $w\in\Wa_\aa$, 
 we define  
\begin{equation}\label{V}
V(w(\l+\rhat_{\aa\,\s})-\rhat_{\aa\,\s})=V(\l)\end{equation} and 
\begin{equation}\label{Vs}V^{sgn}(w(\l+\rhat_{\aa\,\s})-\rhat_{\aa\,\s})=(-1)^{\ell(w)}V(\l).\end{equation}
With this convention we can rewrite \eqref{hwk} as
\begin{equation}\label{hwknominimal}
L(\L)\otimes F^{\s}(\p)^+-L(\L)\otimes F^{\s}(\p)^-=
\sum_{x\in \Wa_\aa\backslash\Wa}(-1)^{\ell(w_x)}V^{sgn}(\varphi^*_\aa(w_x(\L+\rhat_\si))-\rhat_{\aa\si}),
\end{equation}
where $w_x$ is any element from the coset  $x$.

Let  $M\subset \h_0$
be  the lattice defined in \eqref{latticeM}  and set $M_0=\aa^{\ov 0}_0\cap M$ (recall that $\aa^{\ov 0}_0$ is the center of
$\aa^{\ov 0}$). Let $P_0$ be the lattice in $\aa^{\ov 0}_0$ dual to $M_0$. Let $T_{M_0}=\{t_\a\mid\a\in M_0\}$.  Let $\Wa'_{fin}$ be a set of representatives for the cosets of $T_{M_0}\times \Wa_\aa$ in $\Wa$. Assume that 
\begin{equation}\label{center}\aa^{\ov 0}_0=Span_{\C}M_0,\end{equation}
In \cite[Proposition 5.7]{KMP} we proved that $\Wa'_{fin}$ is finite. The following result was also stated in \cite{KMP}, but the proof there wasn't quite 
correct, so we provide here a corrected proof.
\begin{prop}\label{asdim} If $\aa^{\bar0}$ is a reductive 
equal rank subalgebra of
$\g^{\bar0}$ satistfying \eqref{center} and $V^{sgn}(\varphi^*_\aa(w(\L+\rhat_\si))-\rhat_{\aa\si})$ is as in \eqref{Vs}, then
 $$
\sum_{w\in \widehat
W'_{fin}}(-1)^{\ell(w)}\asdim(V^{sgn}(\varphi^*_\aa(w(\L+\rhat_\si))-\rhat_{\aa\si}))=0.
$$
\end{prop}
\begin{proof} 
%Let $M_\aa$ 
%be  the lattice which indexes 
% the translations in the Weyl group
% of $\widehat L(\aa,\s)$. Then,   
%$$\rank( M_0\oplus M_\aa)=\rank\, \aa^{\ov 0}
%=\dim \h_0=\rank\, M, $$
%hence $T_{M_0}\times\Wa_\aa$
% has finite index in $\Wa$. This proves the first claim.
Set $\Dap(\p)=\Dp(\p^{\ov 0})\cup\{j\d+\a\mid j>0, \ \a\in\D(\p^{\ov j})\}$. Then, according to \eqref{vector} and \eqref{weight}, we have
$$
ch_{F^\s(\p)}=e^{\rhat_\s-\rhat_{\aa\s}}\prod_{\a\in\Dap(\p)}(1+e^{-\a})^{\text{mult}\a}.
$$
Hence, setting $ch^\pm=ch_{F^\s(\p)^\pm}$, we have that
$$(ch^+-ch^-)(it,ith)=\left(e^{\rhat_\s-\rhat_{\aa\s}}\prod_{\a\in\Dap(\p)}(1-e^{-\a})^{\text{mult}\a}\right)(it,ith).
$$
Choose now any $\be\in\Dap(\p)$ such that $\be_{|\h_0}\ne 0$. Then we can find $h\in\h_0$ such that  $\be(d+h)=0$ so that, for this particular choice of $h$, we have that
\begin{equation}\label{ch+-ch-}
(ch^+-ch^-)(it,ith)=0.
\end{equation}

%By \eqref{hwk}, we can write
%\begin{align*}0&=ch(L(\L))(it,th)(ch^+-ch^-)(it,th)\\
%&=\sum_{w\in\widehat W'}(-1)^{\ell(w)}
%V(\varphi^*_\aa(w(\L+\rhat_\si))-\rhat_{\aa\si})(it,th).
%\end{align*}

Define $m_w=w(\L+\rhat_\si)(d)$ and $\L^w=w(\L+\rhat_\si)$. 
  Recall that, for $\a\in M_0$, $t_\a(\l)=\l+\l(K)\nu(\a)-((\l(\a)+\half|\a|^2)\d$. Hence  we have
\begin{align*}\varphi^*_\aa(t_\a \L^w)&=
\varphi^*_\aa(\L^w)_{(0)}+(k+g)\a-
((\varphi^*_\aa(\L^w)_{(0)}(\a)+\half|\a|^2)\d_\aa
\\&+\varphi^*_\aa(\L^w)_{(1)}+m_w\d_\aa.
\end{align*}
Setting, for $\l\in\ha^*_{(0)}$,  $\dot{t}_\a(\l)=
\l+(k+g)\nu(\a)-(\l(\a)+\half|\a|^2)\d_\aa$ we can write
$$V^{sgn}(\varphi^*_\aa(t_\a \L^w)-\rhat_{\aa\si})\!=\!
V(\dot t_\a(\varphi^*_\aa(\L^w)_{(0)}))\otimes 
V^{sgn}(\varphi^*_\aa(\L^w)_{(1)}+m_w\d_\aa-\rhat_{\aa\s}).
$$
Since $T_{M_0}\Wa'_{fin}$ is a set of coset representatives for $\Wa_\aa$ in $\Wa$ and observing that multiplying any element $w\in\Wa$ by a translation does not change
the parity of $\ell(w)$ 
we can write, using
\eqref{hwknominimal},
\begin{align*}
&ch(L(\L)\otimes F^{\s}(\p)^+)-ch(L(\L)\otimes F^{\s}(\p)^-)=\\
&\sum_{w\in \widehat 
W'_{fin}}\sum_{\a\in M_0} (-1)^{\ell(t_\a w)}ch(V^{sgn}(\varphi^*_\aa(t_\a \L^w)-\rhat_{\aa\si}))=\\
&\sum_{w\in \widehat 
W'_{fin}} (-1)^{\ell(w)}(\sum_{\a\in M_0}ch(V(\dot t_\a(\L^w_{(0)})))ch(V^{sgn}(\varphi^*_\aa(\L^w)_{(1)}-\rhat_{\aa\si}+m_w\d_\aa))).
\end{align*} 

Set
\begin{equation}\label{thetas}
\varphi_\s=\prod_{\ov j }(\prod_{j\in \ov j,j<0}(1-e^{j\d_\aa}))^
{dim\aa_0^{\bar j}},\qquad \Theta(\l)=\sum\limits_{\a\in M_0}e^{\dot{t}_\a(\l)}.
\end{equation}
Observe that $ch(V(\l_{(0)}))=\frac{e^{\l_{(0)}}}{\varphi_\s }$. Thus
we can write
\begin{align*}
&ch(L(\L)\otimes F^{\s}(\p)^+)-ch(L(\L)\otimes F^{\s}(\p)^-)=
\\&\sum_{w\in \widehat W'_{fin}} (-1)^{\ell(w)}
\frac{\Theta(\L^w_{(0)})}{\varphi_\s}ch(V^{sgn}(\varphi^*_
\aa(\L^w)_{(1)}-\rhat_{\aa\si}+m_w\d_\aa)).
\end{align*} 
By \eqref{hwk}, evaluating both sides at the point $(it,iht)$, we obtain 
from \eqref{ch+-ch-} that
\begin{align*}0&=ch(L(\L))(it,ith)(ch^+-ch^-)(it,th)=\\&\sum_{w\in 
\widehat W'_{fin}} (-1)^{\ell(w)}
\frac{\Theta(\L^w_{(0)})(it,iht)}{\varphi_\s(it)}
ch(V^{sgn}(\varphi^*_\aa(\L^w)_{(1)}-\rhat_{\aa\si}+m_w\d_\aa))(it,ith).
\end{align*} 
Cancelling  $\varphi_\s$ and multiplying by $t^{r/2} e^{-\pi c/12t}$,
where 
\begin{equation}\label{c}
c=\sum\limits_{j=1}^s(\sum_ic_{ji}(k_j))\end{equation} and $r=\dim\aa^{\ov 0}_0$, we obtain: 
$$
 0=\!\lim_{t\to 0^+}\!\!\sum_{w\in \widehat 
W'_{fin}}\!\! (-1)^{\ell(w)}t^{r/2}\Theta(\L^w_{(0)})e^{-\frac{\pi c }{12t}}
ch(V^{sgn}(\varphi^*_\aa(\L^w)_{(1)}-\rhat_{\aa\si}+m_w\d_\aa))(it,ith)
$$
Hence, by \eqref{dimensioneasintotica} and the asymptotics of the theta function (see e.g. 
\cite[(13.13.4)]{Kac}), we find that
$$0=\sum_{w\in\Wa'_{fin}}(-1)^{\ell(w)}|P_0/(k+g)M_0|^{-\half}\asdim(V^{sgn}(\varphi^*_\aa(\L^w)_{(1)}+m_w\d_\aa-\rhat_{\aa\si})).$$
Since, by definition, $$\asdim(V^{sgn}(\varphi^*_\aa(\L^w)_{(1)}+m_w\d_\aa-\rhat_{\aa\si}))=
\asdim(V^{sgn}(\varphi^*_\aa(\L^w)-\rhat_{\aa\si})),$$ 
we are done.
\end{proof}
\vskip10pt 
We now provide a formula affording, as a special case,  the sum of the asymptotic dimension of multiplets
which occur in the decomposition of the basic+vector and spin representation of $\widehat L(so(\p),Ad(\s))$.\vskip5pt
We need preliminarily to sum up the discussion of Section 6 and to add the information about the asymptotic dimension of $F^{\si}(\p)$, 
which can be obtained from  \cite[2.2]{KacW}. We collect all these data in Table 1 where we denote  by $\tilde\L_0,\tilde\L_1,\ldots.\tilde\L_l$ the fundamental weights of
$\widehat L(so(\p),Ad(\s))$.
\vskip10pt
\setlength\extrarowheight{4pt}
\begin{tabular}{| c | c  | c  | c  | c  | c |}
\hline
$\dim(\h_\p)$ & $\dim(\p^{\ov{1/2}})$ &  $Ad(\si)$ & $Ad(\tilde\si)$ & $F^{\si}(\p)$ & $\asdim$\\\hline
even  & even & inner & inner & $V(\tilde\L_{l-1})\oplus V(\tilde\L_l)$ & $1$\\\hline
odd  & odd & not inner & inner &  $2V(\tilde\L_0)$ & $\sqrt{2}$\\\hline
even  & odd & inner & not inner & $V(\tilde\L_{0})\oplus V(\tilde\L_1)$  & $1$\\\hline
odd  & even & inner & inner &  $2 V(\tilde\L_{l})$ & $\sqrt{2}$\\
\hline
\end{tabular}
$$\text{\small Table 1}$$
\begin{prop} Assume  $\L=0$. Let  $\aa$ be the fixed point set of an involution of $\g$, and let $V(\varphi^*_\aa(w(\rhat_\si))-\rhat_{\aa\si})$ be  as in \eqref{V}. Then we have
\begin{equation}\label{sumasdim}
\sum_{w\in \widehat
W'_{fin}}\asdim(V(\varphi^*_\aa(w(\rhat_\si))-\rhat_{\aa\si}))=\sqrt{\frac{|P_0/gM_0|}{2^{\rank(\g^{\ov 0})-\rank(\aa^{\ov 0})-\chi}}}
\end{equation}
where $\chi$ has value $0$ or $1$ according to whether $\si_{|\aa_0}=I_{\aa_0}$ or not.
\end{prop}
\begin{proof} By Proposition \ref{wholeker} below, we have that $Ker(D)=F^{\si}(\p)$. Taking the character of both sides of 
\eqref{kerd} with $\L=0$ and arguing as in the proof of Proposition \ref{asdim} (with the same notation) we get 
\begin{align}\notag&ch(F^{\si}(\p))=\\
\notag&=2^{\lfloor \frac{\rank(\g^{\ov 0}))-
\rank(\aa^{\ov 0})+1}{2}\rfloor}\sum_{w\in \Wa'}ch(V(\varphi_{\aa}^*(w(\rhat_\si))-\rhat_{\aa\,\si}))\\
\notag&=2^{\lfloor \frac{\rank(\g^{\ov 0}))-
\rank(\aa^{\ov 0})+1}{2}\rfloor}
\sum_{w\in \widehat W'_{fin}}\sum_{\a\in M_0}ch(V(\varphi^*_\aa(t_\a \L^w)-\rhat_{\aa\si}))\\
\label{finale}&=2^{\lfloor \frac{\rank(\g^{\ov 0}))-\rank(\aa^{\ov 0})+1}{2}\rfloor}\sum_{w\in \widehat W'_{fin}} 
\frac{\Theta(\L^w_{(0)})}{\varphi_\s}ch(V(\varphi^*_
\aa(\L^w)_{(1)}-\rhat_{\aa\si}+m_w\d_\aa)),\end{align}
where $\varphi_\si$ and $\Theta$ are as in \eqref{thetas}. 

Since the form $(\cdot,\cdot)$ is nondegenerate when restricted to $\aa_0$ we have that $\dim \aa_0^{\bar j}=\dim\aa_0^{-\bar j}$. On the other hand, since $(\g,\aa)$ is symmetric, $\dim \aa_0=1$, so
$$
\varphi_\s=\begin{cases}
\prod_{n\in\ganz^+,n>0}(1-e^{-n\d_\aa})\quad&\text{if $\s_{|\aa_0}=I_{\aa_0}$,}\\
\prod_{n\in\ganz^+}(1-e^{(-n-\half)\d_\aa})&\text{if $\s_{|\aa_0}=-I_{\aa_0}$.}
\end{cases} 
$$
Therefore, if  $\eta$ is the ordinary Dedekind eta function
\cite[(12.2.4)]{Kac}, using   \cite[Ex. 13.2]{Kac}, we deduce that 
$$\varphi_\si(it)=\begin{cases}e^{\pi t/12}\eta(it)\qquad&\text{if $\s_{|\aa_0}=I_{\aa_0}$,}\\
e^{\pi t/24}\frac{\eta(it/2)}{\eta(it)}&\text{if $\s_{|\aa_0}=-I_{\aa_0}$.}\\
\end{cases}
$$ 
 Also remark that $c$ (see \eqref{c}) is the central charge both of the l.h.s. of \eqref{finale}  as a 
$\widehat L(so(\p),Ad(\sigma))$-module and of the r.h.s. as a $\widehat L(\aa,\sigma)$-module. Therefore, evaluate both sides of \eqref{finale} at the point $(it,iht)$
and multiply by $e^{-\frac{\pi c t}{12}}$. Since we know the asymptotics of $\Theta$ and $\eta$ (see formulas (13.13.4), (13.13.5) of \cite{Kac}),  in the limit $t\to 0$ we obtain 
\begin{align}\notag &\asdim(F^{\si}(\p))=\\\label{parziale} &2^{\lfloor \frac{\rank(\g^{\ov 0}))-
\rank(\aa^{\ov 0})+1}{2}\rfloor-\frac{\chi}{2}}\ |P_0/gM_0|^{-\half}\sum_{w\in \widehat W'_{fin}} \asdim(V(\varphi^*_\aa(w(\rhat_\si))-\rhat_{\aa\si})).\end{align}
Plugging into \eqref{parziale} the values for the asymptotic dimension of the l.h.s. obtained in Table 1  we readily get
\eqref{sumasdim}.
\end{proof}
We finally prove that those treated in the previous proposition are the only  instances in which the kernel of the Dirac operator is the whole space $X=L(\L)\otimes F^{\sigma}(\p)$.
We shall freely use the language of vertex algebras. All the necessary information can be found in \cite{KMP}.
\begin{prop}\label{wholeker}  $D$ vanishes identically on $X$ if and only if $\L=0$ and $\g=\aa\oplus\p$ is the Cartan decomposition of an involution of $\g$.
\end{prop}
\begin{proof} Recall  that
in \cite[Remark 4.1]{KMP} we introduced two fields $G,\,L$, depending on $\g,\aa$, in the super affine vertex algebra of $\g$. The field $G=\sum\limits_{n\in\ganz}G_n z^{-n-\frac{3}{2}}$ is related 
to the Dirac operator $D$ by
$G_0^X=\frac{D}{\sqrt{k+g}}$. Assume that the stated condition holds. By Proposition \cite{KMP}, $G_0^X$ splits as the sum of a quadratic and a cubic term (see \cite[(4.1)]{KMP}). The former 
vanishes in $X$ because $\L=0$, the latter vanishes identically because the pair  $(\g,\aa)$ is a symmetric.
Now we prove the converse statement. 
The operators $G,L$  generate a Ramond algebra, i.e., a Lie conformal superalgebra $\C[T]L+\C[T]G+\C C,$ with
\begin{equation}\label{ramond}[L_\l L]=(T+2\l)L+\frac{\l^3}{12}C,\quad [L_\l 
G]=(T+\frac{3}{2}\l)G,\quad [G_\l G]=2L+\frac{\l^2}{3}C\end{equation}
which acts on $X$ with central charge
\begin{align}\label{cc}
C&=\frac{k\dim(\g)}{k+g}-\sum_S(1-\frac{g_S}{k+g})\dim(\aa_S)+\half\dim(\p)\\\label{forces}
&=C(\g)-C(\aa)+\half\dim(\p),
\end{align} 
where $C(\g)=\frac{k\dim(\g)}{k+g}$ is the conformal anomaly of $\g$ and $C(\aa)$ that of $\aa$: see \cite[(4.26)]{KMP} (the first summand in the r.h.s of \eqref{cc} is erroneously missing in the reference). 
 If $Ker(D)=X$, then 
$G_0$ acts trivially. Since the Ramond algebra is simple (modulo the center), 
$G, L$ act trivially and their vanishing implies   that  $C=0$. It is a general fact that if $k>0$ then $C(\g)\geq C(\aa)$ (see the discussion in Section 2 of \cite{AGO}). Since we are assuming that  $L(\L)$ is an integrable module, we have $k\geq 0$. Hence relation \eqref{forces}Ê forces $k=0$, so that 
$$\frac{\dim(\p)}{2}=\sum_S(1-\frac{g_S}{g})\dim(\aa_S).$$
 In turn, this relation tells us that  the Symmetric Space Theorem 
 \cite{GNO} applies, hence $(\g,\aa)$ is a symmetric pair.

\end{proof}

\providecommand{\bysame}{\leavevmode\hbox to3em{\hrulefill}\thinspace}
\providecommand{\MR}{\relax\ifhmode\unskip\space\fi MR }
\providecommand{\MRhref}[2]{%
  \href{http://www.ams.org/mathscinet-getitem?mr=#1}{#2}
}
\providecommand{\href}[2]{#2}

\footnotesize{

\noindent{\bf V.K.}: Department of Mathematics, Rm 2-178, MIT, 77 
Mass. Ave, Cambridge, MA 02139;\\
{\tt kac@math.mit.edu}

\noindent{\bf P.MF.}: Politecnico di Milano, Polo regionale di Como, 
Via Valleggio 11, 22100 Como,
ITALY;\\ {\tt pierluigi.moseneder@polimi.it}

\noindent{\bf P.P.}: Dipartimento di Matematica, Sapienza Universit\`a di Roma, P.le A. Moro 2,
00185, Roma , ITALY;\\ {\tt papi@mat.uniroma1.it} }

\end{document}